\documentclass[1996/10/24]{amsart}
\usepackage{color}
\usepackage{amssymb,stmaryrd}
\usepackage{graphicx}

\newtheorem{theorem}{Theorem}[section]
\newtheorem{algorithm}{Algorithm}[section]
\newtheorem{lemma}[theorem]{Lemma}

\theoremstyle{definition}

\newtheorem{example}[theorem]{Example}

\theoremstyle{remark}

\numberwithin{equation}{section}
\numberwithin{figure}{section}

\usepackage{color}
\definecolor{black}{rgb}{0,0,0}

\definecolor{red}{rgb}{1,0,0}

\definecolor{blue}{rgb}{0,0,1}

\def\bfn{\textit{\textbf{n}}}

\def\bfu{\textit{\textbf{u}}}

\def\bfI{\textit{\textbf{I}}}

\def\bfK{\textit{\textbf{K}}}

\def\bfK{\boldsymbol{K}}
\def\bfI{\boldsymbol{I}}

\begin{document}
\title[Enriched Petrov-Galerkin method for flow in porous media]{A locally mass-conservative enriched Petrov-Galerkin method without penalty for the Darcy flow in porous media}

\author{Huangxin Chen}
 \address{School of Mathematical Sciences and Fujian Provincial Key Laboratory on Mathematical Modeling and High Performance Scientific Computing, Xiamen University, Fujian, 361005, China}
 \email{chx@xmu.edu.cn}

\author{Piaopiao Dong}
 \address{School of Mathematical Sciences and Fujian Provincial Key Laboratory on Mathematical Modeling and High Performance Scientific Computing, Xiamen University, Fujian, 361005, China}
 \email{dongpiaopiao@stu.xmu.edu.cn}

\author{Shuyu Sun$^\ast$}
 \address{Computational Transport Phenomena Laboratory, Division of Physical Science and Engineering, King Abdullah University of Science and Technology, Thuwal, Kingdom of Saudi Arabia}
 \email{shuyu.sun@kaust.edu.sa}

\author{Zixuan Wang}
 \address{School of Mathematical Sciences and Fujian Provincial Key Laboratory on Mathematical Modeling and High Performance Scientific Computing, Xiamen University, Fujian, 361005, China}
 \email{wangzixuan@stu.xmu.edu.cn}

\thanks{$^\ast$Corresponding author: Shuyu Sun (shuyu.sun@kaust.edu.sa)}
\thanks{The work of Huangxin Chen is supported by the NSF of China (Grant No. 12122115, 11771363). The work of Shuyu Sun is supported by King Abdullah University of Science and Technology (KAUST) through the grants BAS/1/1351-01 and URF/1/5028-01.}



\date{}

\begin{abstract}
In this work we present an enriched Petrov-Galerkin (EPG) method for the simulation of the Darcy flow in porous media. The new method enriches the approximation trial space of the conforming continuous Galerkin (CG) method with bubble functions and enriches the approximation test space of the CG method with piecewise constant functions, and it does not require any penalty term in the weak formulation. Moreover, we propose a framework for constructing the bubble functions and consider a decoupled algorithm for the EPG method based on this framework, which enables the process of solving pressure to be decoupled into two steps. The first step is to solve the pressure by the standard CG method, and the second step is a post-processing correction of the first step. Compared with the CG method, the proposed EPG method is locally mass-conservative, while keeping fewer degrees of freedom than the discontinuous Galerkin (DG) method. In addition, this method is more concise in the error analysis than the enriched Galerkin (EG) method. The coupled flow and transport in porous media is considered to illustrate the advantages of locally mass-conservative properties of the EPG method. We establish the optimal convergence of numerical solutions and present several numerical examples to illustrate the performance of the proposed method.
\end{abstract}

\maketitle

%%%%%%%%%%%%%%%%%%%%%%%%%%%%%%%%%%%%%%%%%%%%%%%%%%%%%%%
%----------------section 1-----------------------------
\section{Introduction}
%%%%%%%%%%%%%%%%%%%%%%%%%%%%%%%%%%%%%%%%%%%%%%%%%%%%%%%
The standard continuous Galerkin (CG) finite element method has been widely applied for the solution of second-order elliptic problems for its simplicity in implementation. However, the CG method possesses limitations in local mass conservation, which is significant in numerous applications, particularly when the coupling of flow and transport is considered \cite{Dawson2004,Zhao2023}. Indeed, violation of local mass conservation in velocity fields could cause spurious sources and sinks to the transport simulations, which might lead to numerical inaccuracy for the solution of transport equations. 

Lots of numerical methods preserving local mass conservation have been proposed in the literature, which include the mixed finite element method \cite{Arbogast2002,Arbogast2004,Arbogast1995,Arbogast1996,Wheeler2006}, the finite volume method \cite{Cai1991,Chatzipantelidis2005,Huang1998,Kobayashi1999,Monteagudo2004}, the discontinuous Galerkin (DG) \cite{Arnold1982,Babuska1973,Chen2004,Cockburn2009,Karakashian2003,Larson2004,Riviere1999,Sun2005_2} method, the recent weak Galerkin method \cite{LTW2018}, the mixed virtual element method \cite{Dassi2022} and the reduced basis method \cite{Boon2023}. Besides, many other different finite element schemes based on the DG method have been derived, such as the local discontinuous Galerkin method \cite{Castillo2000,Cockburn1998}, the symmetric interior penalty Galerkin method \cite{Sun2005,Wheeler1978}, the nonsymmetric interior penalty Galerkin method \cite{Riviere2001,Suli2007}, the incomplete interior penalty Galerkin method \cite{Dawson2004,Liu2009,Sun2005}, the Oden-Babu\v{s}ka-Baumann-DG method \cite{Oden1998,Sun2007} and the hybridizable discontinuous Galerkin method \cite{Cockburn2009}. They can tackle rough coefficient cases and capture the nonsmooth features of the solution very well since the function space itself is naturally discontinuous. However, the main disadvantage of the primal DG method is that it has a very large number of degrees of freedom, which can result in expensive computational cost.

In addition, quite a few works have also constructed alternative processing strategies to preserve the local mass conservation \cite{Chippada1998,Cockburn2007,Hughes2000,Larson2004_2,Sun2006}. Among these methods, the enriched Galerkin (EG) method \cite{Lee2017,Lee2018,Lee2016,Sun2009} precisely enriches the CG finite element space with piecewise constant functions, which has fewer degrees of freedom than the DG method while maintaining local mass conservation. Due to the fact that the EG method is essentially a special DG method, its theoretical analysis is relatively similar to the DG method.

In this paper, our objective is to propose a new enriched Petrov-Galerkin (EPG) method for the simulation of the Darcy flow in porous media. The novelties of this work are as follows:

(I) We design a finite element method which enriches the approximation trial space for the CG method with bubble functions and enriches the approximation test space for the CG method with piecewise constant functions respectively. The main difference between the proposed enrich Petrov-Galerkin (EPG) method and the EG method is the selection of the trial functions. Since the bubble function is zero on the boundary of each element, the trial function is still continuous. Unlike the DG method or the EG method, we can still make use of the analysis framework for the conforming finite element method, which greatly simplifies the analysis compared with the DG and EG methods. Meanwhile, the EPG method and the EG method have the same degree of freedom because the degree of freedom for the bubble function on each element is only one. Moreover, the continuity of trial space results in the vanishment of the jump on the element boundaries. Therefore, there is no need to add penalty terms on the interior edges or faces like the DG method or the EG method. 

(II) We present a framework to construct the bubble functions and propose a decoupled algorithm based on the EPG method. The first step in solving pressure is to use the classical CG method, which is easy to obtain. The second step is to modify the numerical solution from the CG method, which is regarded as a post-processing correction. It is clear that this algorithm achieves local mass conservation through the second step.

(III) We consider the coupled flow and transport in porous media to demonstrate the merit of the proposed EPG method. The local mass conservation for the velocity field in flow can ensure to preserve the positivity property of concentration in the transport. We present several interesting numerical examples to illustrate the advantages of the EPG method.

The paper is organized as follows. In Section \ref{section2}, we introduce the governing equations of flow and transport in porous media and state the EPG approximation. In Section \ref{section3}, we establish the well-posedness of the EPG method and present the optimal convergence analysis of the numerical solution for Darcy flow. Local mass conservation for the velocity based on the EPG method and the solution of the coupled flow and transport are considered in Section \ref{section4}. We provide several numerical experiments to illustrate the advantages of the EPG method in Section \ref{section5}. In Section \ref{section6}, we summarize and discuss various features of the EPG method and conclude this paper with possible future work.

%%%%%%%%%%%%%%%%%%%%%%%%%%%%%%%%%%%%%%%%%%%%%%%%%%%%%%%
%----------------section 2-----------------------------
\section{Mathematical model and the EPG method}\label{section2}
%%%%%%%%%%%%%%%%%%%%%%%%%%%%%%%%%%%%%%%%%%%%%%%%%%%%%%%
%---------------section 2.1-----------------------------
\subsection{Governing equations}
In this work, we consider the coupled Darcy flow and species transport within a single flowing phase in porous media. The velocity computed from the Darcy flow will be used for simulations of the species transport in the porous media to solve for the concentration.

For the Darcy flow, it has the following form:
\begin{align}
     -\nabla \cdot (\bfK\nabla p) = \nabla \cdot \bfu &= f, \qquad  \ \  {\rm in}\ \Omega,  \label{pde_cons}\\
     p  &= p_D,   \qquad {\rm on}\ \Gamma_D, \label{pde_bd1} \\
         -\bfK\nabla p \cdot \bfn = \bfu  \cdot \bfn  &= g_N, \qquad  {\rm on}\ \Gamma_N,\label{pde_bd2}
\end{align}
where $\Omega$ is a bounded polygonal/polyhedral domain in $\mathbb{R}^d$ ($d$ = 2 or 3) with boundary $\partial \Omega$, and $\bfn$ denotes the unit outward normal vector on $\partial \Omega$. The unknowns are the pressure $p$ and the velocity $\bfu$. Assume that the conductivity $\bfK$ (the ratio of the permeability and the viscosity) is uniformly symmetric positive definite and bounded from above and below. For simplicity, we assume $\bfK = K \bfI$ in the following, where $\bfI$ is the identity matrix, and $K$ is a piecewisely positive and bounded constant. We impose a Dirichlet boundary condition on $\Gamma_D$ and a Neumann boundary condition on $\Gamma_N$ ($\Gamma_D \cup \Gamma_N=\partial \Omega$, $\Gamma_D \cap \Gamma_N=\emptyset$).

The species transport equation in porous media is described as follows:
\begin{align}
    &\frac{\partial \phi c}{\partial t} + \nabla \cdot (\bfu c)=\widehat{f},&(x,t)\in \Omega \times (0,T_f],\label{transport} \\
    &c = c_I,&(x,t)\in \Gamma_{in} \times (0,T], \\
    &c(x,0) = c_0(x),&x \in \Omega,\label{transport_bc}
\end{align}
where $\phi$ is porosity of porous media, $\widehat{f}$ is the source or sink term, and $T_f$ denotes the final simulation time. Let $\partial \Omega = \Gamma_{in} \cup \Gamma_{out}$ with $\Gamma_{in}=\{x\in \partial\Omega:\bfu \cdot \bfn <0\}$ as the inflow boundary and $\Gamma_{out}=\{x\in \partial\Omega:\bfu \cdot \bfn \geq 0\}$ as the outflow/no-flow boundary. $c_{I}$ denotes the inflow concentration and an initial concentration $c_0$ is specified to close the system.
%---------------section 2.2----------------------------
\subsection{The EPG method}
Let $\mathcal{T}_h$ be a family of partitions of $\Omega$ composed of triangles if $d$ = 2 or tetrahedrons if $d$ = 3. Our method can be extended to quadrilaterals if $d$ = 2 or hexahedra if $d$ = 3, but the space of $\mathbb{P}_k(T)$ below will need to be replaced by $\mathbb{Q}_k(T)$. For convenience of presentation, the following necessary notations are given:
\begin{align*}
    & M^k(\mathcal{T}_h) := \{p \in L^2(\Omega):p|_{T}\in \mathbb{P}_k(T),\quad \forall T\in \mathcal{T}_h\}, \\
    & M_0^k(\mathcal{T}_h) := M^k(\mathcal{T}_h)\cap C^0(\Omega), \\
    & M^b(\mathcal{T}_h) := \{p \in L^2(\Omega):p|_{T}=\alpha_T h(x)\lambda_1 \lambda_2 \cdots\lambda_{d+1},\quad \forall T\in \mathcal{T}_h\},
\end{align*}
where $M^k(\mathcal{T}_h)$ denotes the space of piecewise discontinuous polynomials of degree $k$, while $M_0^k(\mathcal{T}_h)$ consists of continuous polynomials. $M^b(\mathcal{T}_h)$ is the space of piecewise bubble functions and $\alpha_T$ means the degree of freedom.

We now define the average and jump for $v \in H^s(\mathcal{T}_h)$. Let $T_i,T_j \in \mathcal{T}_h$ and $e\in \partial T_i \cap \partial T_j$ with $n_e$ exterior to $T_i$. Denote
\begin{align}\label{jump}
    \{ v \} := \frac{1}{2}\left( v|_{T_i} + v|_{T_j} \right), \quad
    [v] := v|_{T_i} - v|_{T_j}.
\end{align}
We comment that the average $\{ v \}$ or the jump $[v]$ actually means the function value $v$ itself if it is on the Dirichlet boundaries. Next we denote the upwind value of concentration $c^*|_e$ as follows:
\begin{align*}
    c^*|_{e} :=
    \begin{cases}
        c|_{T_i},\quad {\rm if}\  \bfu \cdot \bfn_e \geq 0, \\
        c|_{T_j},\quad {\rm if}\  \bfu \cdot \bfn_e < 0.
    \end{cases}
\end{align*}

We now propose the EPG method for the Darcy flow, which enriches the approximation trial space of the CG method with bubble functions and enriches the approximation test space of the CG method with piecewise constant functions. The trial and test finite element spaces $V_h^1$ and $V_h^2$ for the EPG method are denoted as
\begin{align*}
    & V_h^1 := M_0^k(\mathcal{T}_h)+M^b(\mathcal{T}_h), \\
    & V_h^2 := M_0^k(\mathcal{T}_h)+M^0(\mathcal{T}_h).
\end{align*}
For the Darcy flow, we introduce the bilinear form $a(p_h,q_h)$ and the linear functional $l(q_h)$ as follows:
\begin{align}
    & a(p_h,q_h) = \sum\limits_{T\in \mathcal{T}_h}\int_{T}K\nabla p_h \cdot \nabla q_h dx-\sum\limits_{e\in \Gamma_h \cup \Gamma_{h,D}}\int_{e}\{ K\nabla p_h \cdot \bfn_e \}[q_h]d\sigma - \sum\limits_{e\in \Gamma_h \cup \Gamma_{h,D}}\int_{e}\{ K\nabla q_h \cdot \bfn_e \}[p_h]d\sigma,\label{bilinear} \\
    & l(q_h) = \sum\limits_{T\in \mathcal{T}_h}\int_{T}f q_h dx-\sum\limits_{e\in \Gamma_{h,N}}\int_{e}g_N q_h d\sigma -\sum\limits_{e\in \Gamma_{h,D}}\int_{e} K\nabla q_h \cdot \bfn_e p_D d\sigma,
\end{align}
where the set of all faces ($d=3$) or edges ($d=2$) on interior boundaries, Dirichlet boundaries, Neumann boundaries for $\mathcal{T}_h$ is denoted by $\Gamma_h,\Gamma_{h,D}$ and $\Gamma_{h,N}$, respectively. On each face, a unit normal vector $\bfn_e$ is uniquely chosen. It is obvious that the EPG method uses a weak formulation similar to the EG method, but it does not require any penalty term due to the continuity of the trial space.

The EPG method for the Darcy flow is to seek $p_h \in V_h^1$ satisfying
\begin{align}
    a(p_h,q_h)=l(q_h), \quad \forall q_h \in V_h^2.\label{EPG method}
\end{align}

Now we present a decoupled algorithm for the EPG method, which contains two steps. In the first step, we consider the continuous situation. In the second step, we solve the bubble function, which is regarded as a post-processing correction of the first step.

We decompose $p_h$ and $q_h$ in (\ref{EPG method}) as follows:
\begin{align*}
    &p_h=p_h^c+p_h^b,\quad p_h^c \in M_0^k(\mathcal{T}_h), p_h^b \in M^b(\mathcal{T}_h), \\
    &q_h=q_h^c+q_h^0,\quad q_h^c \in M_0^k(\mathcal{T}_h), q_h^0 \in M^0(\mathcal{T}_h).
\end{align*}
The EPG method can also be rewritten as follows:
\begin{align}
    a(p_h^c,q_h^c) +a(p_h^b,q_h^c)+a(p_h^c,q_h^0)+a(p_h^b,q_h^0)=l(q_h^c)+l(q_h^0).\label{decouple_ori}
\end{align}
If we choose $q_h^c|_{\Gamma_{h,D}}=0$, then we have
\begin{align}
    a(p_h^b,q_h^c) &= \sum\limits_{T\in \mathcal{T}_h}\int_{T}K\nabla p_h^b\cdot \nabla q_h^c dx\notag \\
    &= \sum\limits_{T\in \mathcal{T}_h}\int_{\partial T}  p_h^b( K \nabla q_h^c) \cdot \bfn d\sigma - \sum\limits_{T\in \mathcal{T}_h}\int_{T}p_h^b\nabla \cdot (K \nabla q_h^c)dx\notag \\
    &= - \sum\limits_{T\in \mathcal{T}_h}\int_{T}p_h^b\nabla \cdot (K\nabla q_h^c)dx.\notag
\end{align}
The bubble function $p_h^b$ could be carefully chosen such that $a(p_h^b,q_h^c)=0$. More details on the construction of bubble functions will be provided in the subsequent sections. Now the decoupled algorithm of the EPG method (\ref{EPG method}) is briefly summarized as follows:

\begin{algorithm}
(A decoupled algorithm of the EPG method)

\textbf{Step 1.} Let $q_h^0=0$ and $q_h^c \in M_0^k(\mathcal{T}_h)$. Then (\ref{decouple_ori}) can be simplified as follows: find $p_h^c \in M_0^k(\mathcal{T}_h)$ such that
\begin{align*}
    a(p_h^c,q_h^c)=l(q_h^c),\qquad \forall q_h^c \in M_0^k(\mathcal{T}_h).
\end{align*}

\textbf{Step 2.} Let $q_h^c=0$ and $q_h^0 \in M^0(\mathcal{T}_h)$. Now we can recover $p_h^b$ as follows: find $p_h^b\in M^b(\mathcal{T}_h)$ such that 
\begin{align}\label{EPG_step2}
    a(p_h^b,q_h^0)=l(q_h^0)-a(p_h^c,q_h^0),\qquad \forall q_h^0 \in M^0(\mathcal{T}_h).
\end{align}

\textbf{Step 3.} Get $p_h=p_h^c+p_h^b$.

%\caption{A decoupled algorithm of the EPG method for the Darcy flow in porous media}
\end{algorithm}

We note that the decoupled algorithm of the EPG method divides the solution $p_h$ into two parts. The first part of the solution is the standard solution of the CG method and the well-posedness of $p_h^c$ is well-known. Therefore, in order to verify the well-posedness of $p_h$, we need to focus on the well-posedness of $p_h^b$, which is in the second step of the algorithm. The Babu\v{s}ka theory (cf. \cite{Babuska1972, Brezzi1974, Xu2003}) reveals that the problem (\ref{EPG method}) or (\ref{EPG_step2}) is well-posed if and only if the continuity and coercivity (inf-sup) condition of the bilinear forms are satisfied.

Since the well-posedness of $p^c_h$ is well-known, we will provide the well-posedness of $p_h^b$ in the next section. Now we define the energy norm as follows:
\begin{align}
    \interleave \varphi_h \interleave_{V_h} := \left(\sum\limits_{T\in \mathcal{T}_h}\int_{T}K \nabla \varphi_h \cdot \nabla \varphi_h dx +\sum\limits_{e\in \Gamma_h\cup \Gamma_{h,D}}\frac{1}{h_e}\int_{e} [\varphi_h]^2d\sigma\right)^{1/2}, \forall \varphi_h \in V_h.\label{norm}
\end{align}

The energy norm defined by (\ref{norm}) can be applied to both $V_h^1$ and $V_h^2$ spaces. Since $p_h \in V_h^1$ is globally continuous, the jump term vanishes on $\Gamma_{h}$. Therefore, for any $p_h\in V_h^1$ and $q_h \in V_h^2$, we can describe the energy norms as follows:
\begin{align}
    &\interleave p_h \interleave_{V_h^1} = \left(\sum\limits_{T\in \mathcal{T}_h}\int_{T}K \nabla p_h \cdot \nabla p_h dx +\sum\limits_{e\in \Gamma_{h,D}}\frac{1}{h_e}\int_{e} [p_h]^2d\sigma\right)^{1/2},\label{norm_trial}\\
    &\interleave q_h \interleave_{V_h^2} = \left(\sum\limits_{T\in \mathcal{T}_h}\int_{T}K \nabla q_h \cdot \nabla q_h dx +\sum\limits_{e\in \Gamma_h\cup \Gamma_{h,D}}\frac{1}{h_e}\int_{e} [q_h]^2d\sigma\right)^{1/2}.\label{norm_test}
\end{align}
%%%%%%%%%%%%%%%%%%%%%%%%%%%%%%%%%%%%%%%%%%%%%%%%%%%%%%%
%----------------section 3-----------------------------
\section{Analysis of the EPG method}\label{section3}
%%%%%%%%%%%%%%%%%%%%%%%%%%%%%%%%%%%%%%%%%%%%%%%%%%%%%%%
In this section, we present theoretical analysis of the EPG method, including continuity, coercivity of the bilinear form and the convergence analysis. It should be noted that the positive constant $C$ that appears during the subsequent proofs could have different values but it is depending only the data $K$ and the shape regularity of the meshes.
%---------------------section3.1------------------------------------
\subsection{Continuity of the bilinear forms}
\begin{theorem}\label{continuity}
The bilinear form (\ref{bilinear}) is continuous, i.e.,
\begin{align}\label{continuity_equal}
    |a(p_h,q_h)|\leq C \interleave p_h\interleave_{V_h^1} \interleave q_h\interleave_{V_h^2},\quad \forall p_h \in V_h^1, q_h \in V_h^2,
\end{align}
where $C>0$ is a constant independent of the mesh size $h$.
\end{theorem}
\begin{proof}
For the right-hand side of (\ref{bilinear}), there are three terms which will be analyzed respectively. The first term can be bounded by the H\"{o}lder inequality or the Cauchy-Schwarz inequality as follows:
\begin{align}
    \left| \int_{T}K\nabla p_h \cdot \nabla q_h dx\right| &\leq \|K^{1/2}\nabla p_h\|_{0,T} \|K^{1/2}\nabla q_h\|_{0,T},\notag 
\end{align}
which directly yields
\begin{align}    
    \left| \sum\limits_{T\in \mathcal{T}_h}\int_{T}K\nabla p_h\cdot \nabla q_h dx\right| &\leq \|K^{1/2}\nabla p_h\|_{0,\mathcal{T}_h} \|K^{1/2}\nabla q_h\|_{0,\mathcal{T}_h} \label{cont_1} \\
    &\leq \interleave p_h\interleave_{V_h^1} \interleave q_h\interleave_{V_h^2}.\notag
\end{align}
For the second term, by the H\"{o}lder inequality and the trace inequality, we have
\begin{align*}
    \left| \int_{e}\{ K\nabla p_h \cdot \bfn_e \}[q_h]d\sigma\right| &\leq \|\{ K\nabla p_h \cdot \bfn_e \}\|_{0,e}\|[q_h]\|_{0,e} \\
    &\leq C h_e^{-1/2}\|\nabla p_h\|_{0,w_e}\|[q_h]\|_{0,e} \\
    &\leq C\|K^{1/2}\nabla p_h\|_{0,w_e}\|(\frac{1}{h_e})^{1/2}[q_h]\|_{0,e},
\end{align*}
where $w_e$ represents all elements adjacent to $e$. Therefore, we obtain
\begin{align}
    \left| \sum\limits_{e\in \Gamma_h\cup \Gamma_{h,D}}\int_{e}\{ K\nabla p_h \cdot \bfn_e \}[q_h]d\sigma \right| &\leq C\|K^{1/2}\nabla p_h\|_{0,\mathcal{T}_h}\|(\frac{1}{h_e})^{1/2}[q_h]\|_{0,\Gamma_h\cup \Gamma_{h,D}} \label{cont_2} \\
    &\leq C\interleave p_h\interleave_{V_h^1} \interleave q_h\interleave_{V_h^2}.\notag
\end{align}
We comment that $[q_h]$ actually means the function value $q_h$ instead of the jump if it is on $\Gamma_{h,D}$. The third term can be directly bounded since we use the same norm for both $V_h^1$ and $V_h^2$ spaces. For any $p_h \in V_h^1$, the jump term vanishes on $\Gamma_h$. Now we can conclude the result by (\ref{cont_1}), (\ref{cont_2}) and the triangle inequality.
\end{proof}

In fact, based on the arbitrariness of $p_h$ and $q_h$ in (\ref{continuity_equal}), we can take $p_h$ as $p_h^b$ and $q_h$ as $q_h^0$, which directly shows the continuity condition of the bilinear form in the second step of the decoupled algorithm as follows:
\begin{align}\label{continuity_equal_1}
    |a(p^b_h,q^0_h)|\leq C \interleave p^b_h\interleave_{V_h^1} \interleave q^0_h\interleave_{V_h^2},\quad \forall p^b_h \in M^b(\mathcal{T}_h), q^0_h \in M^0(\mathcal{T}_h).
\end{align}
In order to show the well-posedness of (\ref{EPG_step2}), next we only need to prove the inf-sup condition of $a(p_h^b,q_h^0)$.
%---------------------section3.2-------------------------
\subsection{Coercivity of the bilinear forms}
Denote $p_h^b=\sum\limits_{T\in \mathcal{T}_h}\alpha_T b_T$ where $p_h^b$ is a global bubble function while $b_T$ is its local bubble basis function and $\alpha_T$ is a constant on $T$. Let $\alpha$ be a piecewise constant function with $\alpha|_T = \alpha_T$ for any $T \in \mathcal{T}_h$. The definition of $[\alpha]$ on the face ($d=3$) or the edge ($d=2$) $e$ is based on the previous definition in (\ref{jump}). For any $p^b_h \in M^b(\mathcal{T}_h)$ and $q^0_h \in M^0(\mathcal{T}_h)$, we denote $\interleave  p_h^b \interleave_{M^b(\mathcal{T}_h)} := \interleave  p_h^b \interleave_{V^1_h}$ and $\interleave  q_h^0 \interleave_{M^0(\mathcal{T}_h)} := \interleave  q_h^0 \interleave_{V^2_h}$. Obviously, we have $ \interleave p_h^b\interleave_{M^b(\mathcal{T}_h)} = \|K^{1/2}\nabla p_h^b\|_{0,\mathcal{T}_h}$. Besides, we define another energy norm $\interleave \cdot\interleave_{*}$ for $p_h^b$ by
\begin{align}
    \interleave p_h^b\interleave_{*}:=\left( \sum\limits_{e\in \Gamma_h \cup \Gamma_{h,D}} h_e^{-1} \|[\alpha]\|_{0,e}^2 \right)^{1/2}.\label{norm_star}
\end{align}
\begin{lemma}\label{equivalent}
    The energy norms (\ref{norm_trial}) and (\ref{norm_star}) for $p_h^b$ are equivalent, i.e., for $\forall p_h^b \in M^b(\mathcal{T}_h)$, there hold
    \begin{align*}
        &\interleave  p_h^b \interleave_{M^b(\mathcal{T}_h)} \leq C \interleave  p_h^b \interleave_{*}, \\
        &\interleave  p_h^b \interleave_{*} \leq C \interleave  p_h^b \interleave_{M^b(\mathcal{T}_h)},
    \end{align*}
    where $C>0$ is a constant independent of the mesh size $h$.
\end{lemma}
\begin{proof}
    It is obvious that $\interleave \cdot\interleave_{M^b(\mathcal{T}_h)}$ and $\interleave \cdot\interleave_{*}$ are both energy norms exactly when the set $\Gamma_{h,D}$ is non-empty. Firstly, by the scaling argument and the fact that the norms for the finite dimensional space are equivalent, we have
    \begin{align*}
        \interleave p_h^b\interleave^2_{M^b(\mathcal{T}_h)} &= \|K^{1/2}\nabla p_h^b\|_{0,\mathcal{T}_h} 
        \leq C  \sum\limits_{\hat{T} \in \mathcal{\hat{T}}_h}h_T^{d-2}  \|  \hat{K}^{1/2}\nabla \hat{p}_h^b \|_{0,\hat{T}}^2 
        \leq C\sum\limits_{\hat{e} \in \hat{\Gamma}_h\cup \hat{\Gamma}_{h,D}} h_e^{d-2}   \| [\hat{\alpha}] \|_{0,\hat{e}}^2   \\
        &\leq C  \sum\limits_{e \in \Gamma_h\cup \Gamma_{h,D}}h_e^{-1}\| [\alpha] \|^2_{0,e} 
        = C \interleave  p_h^b \interleave^2_{*},
    \end{align*}
    where the reference constant or function such as $\hat{K}$, $[\hat{\alpha}]$ or $\hat{p}_h^b$ is defined on the reference element $\hat{T}\in \mathcal{\hat{T}}_h$ or the reference face ($d$ = 3) / edge ($d$ = 2) $\hat{e}\in \hat{\Gamma}_h\cup \hat{\Gamma}_{h,D}$. Similarly, we can also have
    \begin{align*}
        \interleave  p_h^b \interleave^2_{*} &=   \sum\limits_{e \in \Gamma_h\cup \Gamma_{h,D}}h_e^{-1}\| [\alpha] \|_{0,e}^2   
         \leq C\sum\limits_{\hat{e} \in \hat{\Gamma}_h\cup \hat{\Gamma}_{h,D}} h_e^{d-2}  \| [\hat{\alpha}] \|_{0,\hat{e}}^2  \notag \\
        &\leq C   \sum\limits_{\hat{T} \in \mathcal{\hat{T}}_h}h_T^{d-2} \|  \hat{K}^{1/2}\nabla \hat{p}_h^b \|_{0,\hat{T}}^2  
        \leq C  \sum\limits_{T\in \mathcal{T}_h} \| K^{1/2}\nabla p_h^b \|_{0,T}^2  
        = C\interleave p_h^b\interleave^2_{M^b(\mathcal{T}_h)}. 
    \end{align*}
    Now we complete the proof.
\end{proof}
%%%% Construction of the bubble functionn %%%%%
Next we start to construct the bubble function such that the orthogonality  $a(p_h^b,q_h^c)=0$ is satisfied for any $p_h^b \in M^b(\mathcal{T}_h)$ and $q_h^0 \in M^0(\mathcal{T}_h)$.   For any $p_h^b=\sum\limits_{T\in \mathcal{T}_h}\alpha_T b_T$, we define $b_T$ in any $T \in \mathcal{T}_h$ as follows:
    \begin{align}\label{bubble_def}
        b_T=
        \begin{cases}
            \sum\limits_{i=1}^{d+1} b_{T,i}, & \quad k=1, \\
            \sum\limits_{i=1}^{d+1} b_{T,i}+\sum\limits_{j=1}^{J} \gamma_j \lambda_1^2 \cdots \lambda_{d+1}^2 \psi_j, & \quad k\geq 2,
        \end{cases}
    \end{align}
    where $J=C_{k+d-2}^d$ and $b_{T,i}$ is a $d$-dimensional one-sided bubble function with the undetermined coefficient $\beta_i$, i.e., $b_{T,i}=\beta_i \frac{\lambda_1^2 \cdots \lambda_{d+1}^2}{\lambda_i}$. Here, $\gamma_j$ is also an undetermined coefficient and $\psi_j$ is a basis function in the barycentric coordinate form of $M^{k-2}(T)$.
    
    We should illustrate that $\beta_i$ and $\gamma_j$ are well-defined. Firstly, we set the undetermined coefficient $\beta_i$ satisfying
    \begin{align}
        \int_{e_i}\beta_i K \nabla \lambda_i \cdot \bfn_i \frac{\lambda_1^2 \cdots \lambda_{d+1}^2}{\lambda_i^2}d\sigma=1.\label{assumption}
    \end{align}
    Thus, $\beta_i$ is well-defined by
    $
        \beta_i=\frac{1}{\int_{e_i} K \nabla \lambda_i \cdot \bfn_i \frac{\lambda_1^2 \cdots \lambda_{d+1}^2}{\lambda_i^2}d\sigma},
    $
    where the denominator can be easily observed as nonzero. Secondly, it is easy to get that $a(p_h^b,q_h^c)=0$ when $k=1$. For $k \geq 2$, we will show that the construction of $b_T$ in (\ref{bubble_def}) satisfies the orthogonality property $a(p_h^b,q_h^c)=0$ if the parameter $\{\gamma_j\}^J_{j=1}$ are well chosen. If $a(p_h^b,q_h^c)=0$ is satisfied, for any $T\in \mathcal{T}_h$, the above $b_T$ satisfies that
    \begin{align}
        (K\sum\limits_{i=1}^{d+1} b_{T,i},\psi_i)_{T}+(\sum\limits_{j=1}^{J} \gamma_j K\lambda_1^2 \cdots \lambda_{d+1}^2 \psi_j,\psi_i)_{T} =0,\quad \forall  \ \psi_i \in M^{k-2}(T),\ i=1,\cdots,J,  \notag
    \end{align}
    where $(\cdot,\cdot)_{T}$ means the inner product on $T$. Then, we get the linear system for $\{\gamma_j\}^J_{j=1}$ as follows:
    \begin{equation}
    \begin{aligned}
        \begin{pmatrix}
            m_{11} & m_{12} & \cdots & m_{1J} \\
            m_{21} & m_{22} & \cdots & m_{2J} \\
            \vdots & \vdots & \ddots & \vdots \\
            m_{J1} & m_{J2} & \cdots & m_{JJ}
        \end{pmatrix}
        \begin{pmatrix}
            \gamma_1 \\
            \gamma_2 \\
            \vdots \\
            \gamma_J
        \end{pmatrix}=
        \begin{pmatrix}
            -(K\sum\limits_{i=1}^{d+1} b_{T,i},\psi_1)_{T} \\
            -(K\sum\limits_{i=1}^{d+1} b_{T,i},\psi_2)_{T} \\
            \vdots \\
            -(K\sum\limits_{i=1}^{d+1} b_{T,i},\psi_J)_{T}
        \end{pmatrix},\nonumber
    \end{aligned}
\end{equation}
where $m_{ij}=(K\lambda_1^2 \cdots \lambda_{d+1}^2 \psi_j,\psi_i)_{T}$. Next, we will show that $\{\gamma_j\}^J_{j=1}$ is well defined by the above linear system. Define a new space $W_h(T)$ as follows:
\begin{align*}
    W_h(T):=\{ \psi \in L^2(T): \psi \in {\rm span} \{ \overline{\psi}_1,\cdots,\overline{\psi}_J \} \}, \quad \forall T\in \mathcal{T}_h,
\end{align*}
where $\overline{\psi}_i := K^{1/2}\lambda_1 \cdots \lambda_{d+1} \psi_i$ is the basis function for $W_h(T)$. The coefficient matrix can be simplified as follows:
\begin{align}
    m_{ij}=(K\lambda_1^2 \cdots \lambda_{d+1}^2 \psi_j,\psi_i)_{T} 
    =(\overline{\psi}_j,\overline{\psi}_i)_{T},\notag
\end{align}
which yields a nonsingular coefficient matrix because $\{ \overline{\psi}_i \}_{i=1}^J$ are linearly independent. Besides, the right-hand term of the above linear system is nonzero, which shows that $\{\gamma_j\}^J_{j=1}$ can be well determined. Thus, we can suitably choose the bubble function such that the orthogonality property $a(p_h^b,q_h^c)=0$ holds true.
%%%%%%%%%%%%%%%%%%%%%%%%%%%

Next we show the coercivity estimate for the bilinear form $a(p_h^b,q_h^0)$.
\begin{theorem}\label{corecivity}
    The following coercivity results hold true:
    \begin{align}
        &\mathop{sup}\limits_{p_h^b \neq 0}\frac{a(p_h^b,q_h^0)}{\interleave p_h^b\interleave_{M^b(\mathcal{T}_h)}}\geq C\interleave q_h^0\interleave_{M^0(\mathcal{T}_h)},\quad \forall q_h^0 \in M^0(\mathcal{T}_h), \label{inf_sup1} \\
        &\mathop{sup}\limits_{q_h^0 \neq 0}\frac{a(p_h^b,q_h^0)}{\interleave q_h^0\interleave_{M^0(\mathcal{T}_h)}}\geq C \interleave p_h^b\interleave_{M^b(\mathcal{T}_h)},\quad \forall p_h^b \in M^b(\mathcal{T}_h),\label{inf_sup2}
    \end{align}
    where $C>0$ is a constant independent of the mesh size $h$.
\end{theorem}
\begin{proof}
By the definition of (\ref{bubble_def}), we have 
\begin{align}\label{bubble_edge}
    \int_{\partial T}K \nabla b_T \cdot \bfn d\sigma
    =\int_{e_1}\beta_1 K \nabla \lambda_1 \cdot \bfn_1 \lambda_2^2 \cdots \lambda_{d+1}^2 d\sigma
    + \cdots
    +\int_{e_{d+1}}\beta_{d+1} K \nabla \lambda_{d+1} \cdot \bfn_{d+1} \lambda_1^2 \cdots \lambda_d^2 d\sigma .
\end{align}
Obviously, each of the above $d+1$ terms on the right-hand side of (\ref{bubble_edge}) equals to 1 according to (\ref{assumption}).

Firstly, for any $q_h^0 \in M^0(\mathcal{T}_h)$, we choose $\alpha_T=-q_h^0|_{T}$. We consider the integral on the face ($d=3$) or the edge ($d=2$) $e$ of an element $T$ as follows:
\begin{align}
    \int_{e} K\nabla p_h^b \cdot \bfn_e d\sigma=\alpha_T \int_{e} K\nabla b_T \cdot \bfn_e d\sigma=\alpha_T,\notag
\end{align}
which yields
\begin{align*}
    \int_{e} \{K\nabla p_h^b \cdot \bfn_e\} d\sigma=
    \begin{cases}
        -\frac{1}{2}[q_h^0],&\quad e \in \Gamma_h, \\
        - q_h^0,&\quad e \in \Gamma_{h,D},
    \end{cases}
\end{align*}
where $\bfn_e$ is uniquely chosen. Therefore, we have
\begin{align}\label{infsup1}
    a(p_h^b,q_h^0) = -\sum\limits_{e\in \Gamma_h\cup \Gamma_{h,D}}\int_{e}\{ K\nabla p_h^b \cdot \bfn_e \}[q_h^0]d\sigma 
     \geq \frac{1}{2} \sum\limits_{e\in \Gamma_h\cup \Gamma_{h,D}}\int_{e}h_e^{-1}[q_h^0]^2d\sigma  
    = \frac{1}{2} \interleave q_h^0\interleave_{M^0(\mathcal{T}_h)}^2. 
\end{align}
By Lemma \ref{equivalent}, we have
\begin{align}\label{infsup2}
    \interleave p_h^b\interleave_{M^b(\mathcal{T}_h)} &\leq C \interleave  p_h^b \interleave_{*} 
    =C\left( \sum\limits_{e \in \Gamma_h\cup \Gamma_{h,D}}h_e^{-1}\| [q_h^0] \|_{0,e}^2 \right)^{1/2} 
    = C \interleave q_h^0\interleave_{M^0(\mathcal{T}_h)}. 
\end{align}
Following (\ref{infsup1}) and (\ref{infsup2}), we directly obtain
\begin{align*}
    \frac{a(p_h^b,q_h^0)}{\interleave p_h^b\interleave_{M^b(\mathcal{T}_h)}} \geq C \frac{a(p_h^b,q_h^0)}{\interleave q_h^0\interleave_{M^0(\mathcal{T}_h)}} \geq \frac{C}{2}\interleave q_h^0\interleave_{M^0(\mathcal{T}_h)},
\end{align*}
which yields the estimate (\ref{inf_sup1}).

Secondly, for any $p_h^b \in M^b(\mathcal{T}_h)$, we choose $q_h^0|_{T}=-\alpha_T$. Similarly, we obtain
\begin{align}\label{infsup3} 
    a(p_h^b,q_h^0) &\geq \frac{1}{2} \sum\limits_{e\in \Gamma_h\cup \Gamma_{h,D}}\int_{e}h_e^{-1}[q_h^0]^2d\sigma
    = \frac{1}{2} \interleave  p_h^b \interleave_{*}^2 
    \geq C \interleave p_h^b\interleave_{M^b(\mathcal{T}_h)}^2. 
\end{align}
By Lemma \ref{equivalent}, we also have
\begin{align}\label{infsup4} 
    \interleave q_h^0\interleave_{M^0(\mathcal{T}_h)}&=\left( \sum\limits_{e \in \Gamma_h\cup \Gamma_{h,D}}h_e^{-1}\| [q_h^0] \|_{0,e}^2 \right)^{1/2}
    = \interleave  p_h^b \interleave_{*}  
    \leq C\interleave p_h^b\interleave_{M^b(\mathcal{T}_h)}. 
\end{align}
By (\ref{infsup3}) and (\ref{infsup4}), we obtain
\begin{align*}
    \frac{a(p_h^b,q_h^0)}{\interleave q_h^0\interleave_{M^0(\mathcal{T}_h)}} \geq C \frac{a(p_h^b,q_h^0)}{\interleave p_h^b\interleave_{M^b(\mathcal{T}_h)}} \geq C \interleave p_h^b\interleave_{M^b(\mathcal{T}_h)},
\end{align*}
which yields the estimate (\ref{inf_sup2}). Now we complete the proof.
\end{proof}
The continuity and the two inf-sup conditions for the bilinear form $a(p_h^b,q_h^0)$ have been proved through Theorem \ref{continuity} and Theorem \ref{corecivity}. Thus, we can directly see that the solution $p_h^b$ of (\ref{EPG_step2}) is well-posed in the second step of the decoupled algorithm by the Babu\v{s}ka theory (cf. \cite{Babuska1972, Brezzi1974, Xu2003}). Since the well-posedness of the solution $p_h^c$ in the first step to solve is well-known, we can easily conclude that the solution $p_h$ of (\ref{EPG method}) for the EPG method is well-posed. By the Babu\v{s}ka theory (cf. \cite{Babuska1972, Brezzi1974, Xu2003}) again, the following two inf-sup conditions for the bilinear form $a(p_h,q_h)$ are satisfied:
\begin{align}
    &\mathop{sup}\limits_{p_h \neq 0}\frac{a(p_h,q_h)}{\interleave p_h \interleave_{V_h^1}}\geq C\interleave q_h\interleave_{V_h^2},\quad \forall q_h \in V_h^2,\notag \\
    &\mathop{sup}\limits_{q_h \neq 0}\frac{a(p_h,q_h)}{\interleave q_h \interleave_{V_h^2}}\geq C\interleave p_h\interleave_{V_h^1},\quad \forall p_h \in V_h^1.\label{infsup_ori2}
\end{align}

%---------------------section3.3-------------------------
\subsection{Convergence analysis}
We now introduce several necessary interpolations, which will be used in the analysis of convergence. Firstly, let $I_h^c:L^1(\Omega)\rightarrow M_0^k(\mathcal{T}_h)$ be the Cl\'{e}ment interpolation operator \cite{Clement1975}. The interpolation error estimates of $I_h^c$ have been well-known in the literature (cf. \cite{Clement1975}) as follows: for any $v \in H^s(\Omega)$, $s>1$, there hold
\begin{align}
    &\|v -I_h^c v\|_{0,\mathcal{T}_h} \leq Ch^s|v|_{s,\Omega},\label{Ihc1} \\
    &|v-I_h^c v|_{1,\mathcal{T}_h} \leq Ch^{s-1}|v|_{s,\Omega}.\label{Ihc2}
\end{align}
Secondly, let $I_h^b:L^2(\Omega)\rightarrow M^b(\mathcal{T}_h)$ be the linear projection that satisfies the constraint as follows:
\begin{align*}
    \int_{T}(I_h^b \varphi - \varphi )dx=0,\quad T \in \mathcal{T}_h, \ \forall \varphi \in L^2(\Omega).
\end{align*}
The estimates for $I_h^b$ can be guaranteed (cf. \cite{Arnold1984}) as follows:
\begin{align}
    &\|I_h^b \varphi \|_{0,\mathcal{T}_h} \leq C\|\varphi \|_{0,\Omega},\quad \forall \varphi \in L^2(\Omega), \label{Ihb1}\\
    &|I_h^b(\varphi-I_h^c \varphi)|_{1,\mathcal{T}_h} \leq Ch^{s-1}|\varphi|_{s,\Omega}, \quad \forall \varphi \in H^s(\Omega),  s>1.\label{Ihb2}
\end{align}
Finally, we define a new interpolation operator $I_h:L^2(\Omega)\rightarrow V_h^1$ as follows: for any $v \in L^2(\Omega)$,
\begin{align}\label{interpolation_def}
    I_h v:=I_h^c v +I_h^b(v-I_h^c v),
\end{align}
which combines $I_h^c$ and $I_h^b$ together. It is exactly the interpolation projecting into the trial space $V_h^1$. Next we will show the error estimates for the interpolation $I_h$ defined in (\ref{interpolation_def}).
\begin{lemma}\label{interpolation}
    The interpolation $I_h$ defined in (\ref{interpolation_def}) satisfies the error estimates as follows: for $v \in H^s(\Omega)$, there hold
    \begin{align}
        &\|v-I_h v\|_{0,\mathcal{T}_h} \leq Ch^s |v|_{s,\Omega},\label{interpolation_lemma1} \\
        &|v-I_h v|_{1,\mathcal{T}_h} \leq Ch^{s-1}|v|_{s,\Omega},\label{interpolation_lemma2} \\
        &\interleave v-I_h v\interleave_{V_h^1} \leq Ch^{s-1}|v|_{s,\Omega},\label{interpolation_lemma3}
    \end{align}
    where $C>0$ is a constant independent of the mesh size $h$.
\end{lemma}
\begin{proof}
    Firstly, we will consider the estimate of $\|v-I_h v\|_{0,\mathcal{T}_h}$. By the triangle inequality, (\ref{Ihc1}) and (\ref{Ihb1}), we have
    \begin{align*}
        \|v-I_h v\|_{0,\mathcal{T}_h} \leq \| v-I_h^c v \|_{0,\mathcal{T}_h}+\| I_h^b(v-I_h^c v) \|_{0,\mathcal{T}_h}
        \leq C\| v-I_h^c v \|_{0,\mathcal{T}_h}
        \leq Ch^s|v|_{s,\Omega},
    \end{align*}
    which yields the estimate (\ref{interpolation_lemma1}).
    
    Secondly, we will estimate $|v-I_h v|_{1,\mathcal{T}_h}$. Following the triangle inequality, (\ref{Ihc2}) and (\ref{Ihb2}), we directly obtain
    \begin{align*}
        |v-I_h v|_{1,\mathcal{T}_h} \leq | v-I_h^c v |_{1,\mathcal{T}_h}+| I_h^b(v-I_h^c v) |_{1,\mathcal{T}_h}
        \leq Ch^{s-1}|v|_{s,\Omega},
    \end{align*}
    which yields the estimate (\ref{interpolation_lemma2}).
    
    Finally, we will take $\interleave v-I_h v\interleave_{V_h^1}$ into consideration. For convenience, we denote $v-I_h v$ by $e_v$. By the trace inequality, (\ref{interpolation_lemma1}) and (\ref{interpolation_lemma2}), we have
    \begin{align*}
        \interleave e_v\interleave_{V_h^1}^2 &= \| K^{1/2}\nabla e_v \|_{0,\mathcal{T}_h}^2+\sum\limits_{e\in \Gamma_{h,D}}\frac{1}{h_e}\| [e_v] \|_{0,e}^2 \\
        &\leq C\| \nabla e_v \|_{0,\mathcal{T}_h}^2 + C\sum\limits_{T\in \mathcal{T}_h}\frac{1}{h_T}\left( h_T^{-1}\| e_v \|_{0,T}^2 +h_T\| \nabla e_v \|_{0,T}^2\right) \\
        &\leq Ch^{2s-2}|v|_{s,\Omega}^2,
    \end{align*}
    which yields the estimate (\ref{interpolation_lemma3}). Now we complete the proof.
\end{proof}
\begin{theorem}\label{errorp}
    Let $p$ be the solution of the problem (\ref{pde_cons})-(\ref{pde_bd2}) and $p_h$ be the solution of  the EPG method (\ref{EPG method}). We assume that $p \in H^s(\Omega)$, where $s>\frac{3}{2}$. Then there exists a constant $C>0$ independent of the mesh size $h$ such that
    \begin{align}\label{errorp_equal}
        \interleave  p-p_h \interleave_{V_h^1} \leq Ch^{s-1}|p|_{s,\Omega}.
    \end{align}
\end{theorem}
\begin{proof}
    For any $q_h \in V_h^2$, we have 
    $
        a(p,q_h)=l(q_h)$  and
      $  a(p_h,q_h)=l(q_h),
    $
    which directly yield that \[a(p-p_h,q_h)=0.\] 
    Following (\ref{continuity_equal}), (\ref{infsup_ori2}) and the above orthogonality property, we have
    \begin{align}
        C\interleave  p_h- I_h p\interleave_{V_h^1} &\leq \mathop{sup}\limits_{q_h \neq 0}\frac{a(p_h-I_h p,q_h)}{\interleave q_h\interleave_{V_h^2}}\label{galerkin} \\
        &\leq \mathop{sup}\limits_{q_h \neq 0}\frac{a(p-I_h p,q_h)}{\interleave q_h\interleave_{V_h^2}}\notag \\
        &\leq C\interleave  p-I_h p \interleave_{V_h^1}.\notag
    \end{align}
    Therefore, by the triangle inequality, (\ref{interpolation_lemma3}) and (\ref{galerkin}), we have
    \begin{align*}
        \interleave  p-p_h\interleave_{V_h^1} &\leq \interleave  p- I_h p\interleave_{V_h^1}+\interleave  I_h p -p_h\interleave_{V_h^1} \\
        &\leq C\interleave  p- I_h p\interleave_{V_h^1} \\
        &\leq Ch^{s-1}|p|_{s,\Omega},
    \end{align*}
    which yields the desired estimate (\ref{errorp_equal}). 
\end{proof}
%%%%%%%%%%%%%%%%%%%%%%%%%%%%%%%%%%%%%%%%%%%%%%%%%%%%%%%
%----------------section 4-----------------------------
\section{Local mass conservation and application to the species transport system}\label{section4}
%%%%%%%%%%%%%%%%%%%%%%%%%%%%%%%%%%%%%%%%%%%%%%%%%%%%%%%
After the discrete solution of pressure in the EPG method is obtained by the decoupled algorithm, the velocity $\bfu_h$ can be recovered as follows:
\begin{align}
    &\bfu_h = - K \nabla(p_h^c + p_h^b),\quad &x \in T,\quad T \in \mathcal{T}_h, \label{conservation_1}\\
    &\bfu_h \cdot \bfn = g_N,\quad &x \in \Gamma_{h,N},\\
    &\bfu_h \cdot \bfn = -\{ K \nabla(p_h^c + p_h^b)\cdot \bfn \}|_{e\in \Gamma_h},\quad &x \in \Gamma_h, \\
    &\bfu_h \cdot \bfn = -( K \nabla(p_h^c + p_h^b)\cdot \bfn )|_{e\in \Gamma_{h,D}},\quad &x \in \Gamma_{h,D}.\label{conservation_4}
\end{align}
\begin{theorem}\label{erroru}
    Let $\bfu$ be the velocity of the Darcy flow in (\ref{pde_cons}) and $\bfu_h$ be the discrete velocity solution defined as in (\ref{conservation_1})-(\ref{conservation_4}). If all the assumptions in Theorem \ref{errorp} hold, then there exists a constant $C>0$ independent of the mesh size $h$ such that
    \begin{align}
        &\| \bfu-\bfu_h \|_{0,\mathcal{T}_h} \leq Ch^{s-1}|p|_{s,\Omega},\label{erroru_equal1} \\
        &\| \bfu \cdot \bfn - \bfu_h \cdot \bfn \|_{0,\Gamma_h \cup \Gamma_{h,D}} \leq Ch^{s-\frac{3}{2}}|p|_{s,\Omega}.\label{erroru_equal2}
    \end{align}
\end{theorem}
\begin{proof}
    Firstly, by (\ref{errorp_equal}), we directly obtain
    \begin{align*}
        \| \bfu-\bfu_h \|_{0,\mathcal{T}_h}^2
        \leq C \interleave  p-p_h \interleave_{V_h^1}^2 
        \leq Ch^{2s-2}|p|_{s,\Omega}^2,
    \end{align*}
    which yields the error estimate (\ref{erroru_equal1}). Next, in order to estimate the error on the interior faces or edges, we let $\overline{p}=I_h p \in V_h^1$ be an interpolation of $p$ and $\overline{\bfu}=-K\nabla \overline{p}$ be the corresponding velocity within each element. On the interface $e=\partial T_i \cap \partial T_j\in \Gamma_h$, $\overline{\bfu}$ is defined as follows:
    \begin{align*}
        \overline{\bfu}|_{e}:=\{ \overline{\bfu} \}
        =\frac{\overline{\bfu}|_{T_i}+\overline{\bfu}|_{T_j}}{2}.
    \end{align*}
    We easily have
    \begin{align}\label{mean_equal}
        \sum\limits_{e\in \Gamma_h}\| \bfu \cdot \bfn_e - \bfu_h \cdot \bfn_e \|_{0,e}^2
        \leq 2\sum\limits_{e\in \Gamma_h}\| \bfu \cdot \bfn_e - \overline{\bfu} \cdot \bfn_e \|_{0,e}^2
        + 2\sum\limits_{e\in \Gamma_h}\| \overline{\bfu} \cdot \bfn_e - \bfu_h \cdot \bfn_e \|_{0,e}^2.
    \end{align}
    For the right-hand side of (\ref{mean_equal}), the two terms will be analyzed respectively. The first term can be bounded by the trace inequality and (\ref{interpolation_lemma2}) as follows:
    \begin{align}\label{trace_conclusion1}
        \sum\limits_{e\in \Gamma_h}\| \bfu \cdot \bfn_e - \overline{\bfu} \cdot \bfn_e \|_{0,e}^2
        \leq C \sum\limits_{T \in \mathcal{T}_h}\frac{1}{h_T}\| \bfu - \overline{\bfu} \|_{0,T}^2
        \leq Ch^{2s-3}|p|_{s,\Omega}^2.
    \end{align}
    Then we can estimate the second term by the trace inequality as follows:
    \begin{align}\label{trace_equal1}
        \sum\limits_{e\in \Gamma_h}\| \overline{\bfu} \cdot \bfn_e - \bfu_h \cdot \bfn_e \|_{0,e}^2
        &\leq C\sum\limits_{\partial T_i \cap \partial T_j = e\in \Gamma_h}\left( \left\|  \overline{\bfu}|_{T_i} - \bfu_h|_{T_i} \right\|_{0,e}^2 +\left\| \overline{\bfu}|_{T_j} - \bfu_h|_{T_j} \right\|_{0,e}^2\right) \\
        &\leq C \sum\limits_{T \in \mathcal{T}_h}\frac{1}{h_T}\| \bfu_h - \overline{\bfu} \|_{0,T}^2.\notag
    \end{align}
    By (\ref{interpolation_lemma2}) and (\ref{erroru_equal1}), the right-hand side of (\ref{trace_equal1}) can be bounded as follows:
    \begin{align}\label{trace_conclusion2}
        \sum\limits_{T \in \mathcal{T}_h}\frac{1}{h_T}\| \bfu_h - \overline{\bfu} \|_{0,T}^2
        \leq C\sum\limits_{T \in \mathcal{T}_h}\frac{1}{h_T}\| \bfu_h - \bfu \|_{0,T}^2
        + C\sum\limits_{T \in \mathcal{T}_h}\frac{1}{h_T}\| \bfu - \overline{\bfu} \|_{0,T}^2
        \leq Ch^{2s-3}|p|_{s,\Omega}^2.
    \end{align}
    By (\ref{trace_conclusion1}) and (\ref{trace_conclusion2}), we complete the error estimate (\ref{erroru_equal2}) on the interior faces or edges. Besides, the error estimate on Dirichlet boundaries could be obtained similarly and we complete the proof.
\end{proof}

We note that the local mass conservation is an important property which is significant in numerous applications. The velocity $\bfu_h$ recovered by the EPG method is locally mass-conservative, which will be shown in the following.
\begin{theorem}
    The velocity $\bfu_h$ defined in (\ref{conservation_1})-(\ref{conservation_4}) is locally mass-conservative, which means
    \begin{align}\label{conservation_equal}
        \int_{\partial T} \bfu_h \cdot \bfn d\sigma = \int_{T} fdx,\quad \forall T \in \mathcal{T}_h.
    \end{align}
\end{theorem}
\begin{proof}
    We take a local element $T \in \mathcal{T}_h$ into consideration. For the weak formulation (\ref{EPG method}), the test function $q_h$ is chosen as follows:
    \begin{align*}
        q_h :=
        \begin{cases}
            &1,\quad x \in T, \\
            &0,\quad else,
        \end{cases}
    \end{align*}
    which directly yields (\ref{conservation_equal}). This completes the proof.
\end{proof}

Next we take the recovered velocity $\bfu_h$ into (\ref{transport}) to simulate the concentration in porous media. The scheme used to discrete the time derivative term could be explicit or implicit and the finite element space for concentration is chosen as $M^0(\mathcal{T}_h)$.

For the species transport equation, we introduce two fully discrete schemes by choosing different time discrete schemes as follows: Given $c^{n-1}_h \in M^0(\mathcal{T}_h)$, find  $c^{n}_h \in M^0(\mathcal{T}_h)$ such that
\begin{align}
    \sum\limits_{T\in \mathcal{T}_h}\int_T \phi \frac{c_h^n-c_h^{n-1}}{\Delta t} q_h dx
    &+\sum\limits_{e\in\Gamma_h}\int_e c_h^{n-1,*}\bfu_h \cdot \bfn_e [q_h]d\sigma\label{explicit transport scheme} \\
    +\sum\limits_{e\in\Gamma_{in}}\int_e c_I \bfu_h \cdot \bfn_e q_h d\sigma
    &+\sum\limits_{e\in\Gamma_{out}}\int_e c_h^{n-1} \bfu_h \cdot \bfn_e q_h d\sigma=\sum\limits_{T\in \mathcal{T}_h}\int_T \widehat{f} q_h dx,\quad \forall q_h \in M^0(\mathcal{T}_h),\notag
\end{align}
or
\begin{align}
    \sum\limits_{T\in \mathcal{T}_h}\int_T \phi \frac{c_h^n-c_h^{n-1}}{\Delta t} q_h dx
    &+\sum\limits_{e\in\Gamma_h}\int_e c_h^{n,*}\bfu_h \cdot \bfn_e [q_h]d\sigma\label{implicit transport scheme} \\
    +\sum\limits_{e\in\Gamma_{in}}\int_e c_I \bfu_h \cdot \bfn_e q_h d\sigma
    &+\sum\limits_{e\in\Gamma_{out}}\int_e c_h^{n} \bfu_h \cdot \bfn_e q_h d\sigma=\sum\limits_{T\in \mathcal{T}_h}\int_T \widehat{f} q_h dx,\quad \forall q_h \in M^0(\mathcal{T}_h),\notag
\end{align}
where $\Delta t$ is the time step size. Here, (\ref{explicit transport scheme}) corresponds to the explicit scheme while (\ref{implicit transport scheme}) corresponds to the implicit scheme. In fact, when the discrete solution of velocity is locally mass-conservative, the discrete solution of concentration for the transport equation can hold the physical properties which will be shown in the following two theorems.

\begin{theorem}\label{explicit thm}
If the time step size $\Delta t$ is small enough, the solution of concentration of the transport equation (\ref{transport})-(\ref{transport_bc}) approximated by the explicit scheme (\ref{explicit transport scheme}) has the positive and upper bound preserving properties, which means
    \begin{align}
        &c_h^{n}\geq0, \qquad {\rm if} \ c_h^{n-1}\geq0,\label{explicit_property1} \\
        &c_h^{n}\leq 1, \qquad {\rm if} \  c_h^{n-1}\leq 1.\label{explicit_property2}
    \end{align}
\end{theorem}
\begin{proof}   
    It suffices to consider the case in a local element $T$. For simplicity, we assume that the source/sink term equals to zero and all the faces ($d=3$) or edges ($d=2$) of $T$ are in $\Gamma_h$. We denote the inflow and outflow interior faces or edges sets of $\partial T$ as follows:
    \begin{align}
        &\mathcal{E}^T_1 := \{ e\in \partial T: \bfu_h \cdot \bfn_e \geq 0\},\label{outflow} \\
        &\mathcal{E}^T_2 := \{ e\in \partial T: \bfu_h \cdot \bfn_e < 0\},\label{inflow}
    \end{align}
where $\partial T \cap \partial \Omega \neq \emptyset$. The local mass conservation on this element $T$ is described as follows:
    \begin{align}\label{explicit_conservation}
        \sum\limits_{e\in \mathcal{E}^T_1}\int_{e}\bfu_h \cdot \bfn_e d\sigma
        + \sum\limits_{e\in \mathcal{E}^T_2}\int_{e} \bfu_h \cdot \bfn_e d\sigma = 0.
    \end{align}
    Let $q_h|_T =1$ and $q_h =0$ on other elements in (\ref{explicit transport scheme}), we have
    \begin{align}
        \int_T \phi \frac{c_h^n-c_h^{n-1}}{\Delta t}dx
        +\sum\limits_{e\in \mathcal{E}^T_1}\int_{e} c_h^{n-1}\bfu_h \cdot \bfn_e d\sigma
        +\sum\limits_{e\in \mathcal{E}^T_2}\int_{e} c_h^{n-1,out}\bfu_h \cdot \bfn_e d\sigma = 0,\label{local_transpose}
    \end{align}
    where $c_h^{n-1,out}$ denotes the concentration on the outflow faces or edges. Since the approximation space for the concentration is $M^0(\mathcal{T}_h)$, following (\ref{local_transpose}) we obtain
    \begin{align}
        c_h^n &= c_h^{n-1}+\frac{\Delta t}{|T|\phi}\left( -\sum\limits_{e\in \mathcal{E}^T_1}\int_{e} c_h^{n-1}\bfu_h \cdot \bfn_e d\sigma - \sum\limits_{e\in \mathcal{E}^T_2}\int_{e} c_h^{n-1,out}\bfu_h \cdot \bfn_e d\sigma \right)\notag \\
        &\geq c_h^{n-1}-\frac{\Delta t}{|T|\phi} \sum\limits_{e\in \mathcal{E}^T_1}\int_{e} c_h^{n-1}\bfu_h \cdot \bfn_e d\sigma\notag \\
        &=\left(1-\frac{\Delta t}{|T|\phi} \sum\limits_{e\in \mathcal{E}^T_1}\int_{e} \bfu_h \cdot \bfn_e d\sigma\right)c_h^{n-1},\notag
    \end{align}
    where $|T|$ denotes the volume ($d=3$) or the area ($d=2$) of $T$. Thus, for any $T \in \mathcal{T}_h \ (\partial T \cap \partial\Omega = \emptyset)$ , if we assume that the time step size $\Delta t$ satisfies the restriction as follows:
    \begin{align}
        \Delta t <  \frac{|T|\phi}{\sum\limits_{e\in \mathcal{E}^T_1}\int_{e} \bfu_h \cdot \bfn_e d\sigma}, ,\label{timestep}
    \end{align}
    then this completes the proof of the positive-preserving property (\ref{explicit_property1}) if $c_h^{n-1}|_T>0$. 

    Next, by (\ref{explicit_conservation}), (\ref{local_transpose}) and (\ref{timestep}), we have
    \begin{align*}
        1-c_h^n &= 1-c_h^{n-1}+\frac{\Delta t}{|T|\phi}\left(-c_h^{n-1}\sum\limits_{e\in \mathcal{E}^T_2}\int_{e} \bfu_h \cdot \bfn_e d\sigma +\sum\limits_{e\in \mathcal{E}^T_2}\int_{e} c_h^{n-1,out}\bfu_h \cdot \bfn_e d\sigma \right) \\
        &= 1-c_h^{n-1}+\frac{\Delta t}{|T|\phi}\left( (1-c_h^{n-1})\sum\limits_{e\in \mathcal{E}^T_2}\int_{e} \bfu_h \cdot \bfn_e d\sigma +\sum\limits_{e\in \mathcal{E}^T_2}\int_{e} (c_h^{n-1,out}-1)\bfu_h \cdot \bfn_e d\sigma \right)\\
        &\geq 1-c_h^{n-1}+\frac{\Delta t}{|T|\phi} (1-c_h^{n-1})\sum\limits_{e\in \mathcal{E}^T_2}\int_{e} \bfu_h \cdot \bfn_e d\sigma \\
        &= (1-c_h^{n-1})\left(1+\frac{\Delta t}{|T|\phi}\sum\limits_{e\in \mathcal{E}^T_2}\int_{e} \bfu_h \cdot \bfn_e d\sigma\right) \\
        &= (1-c_h^{n-1})\left(1-\frac{\Delta t}{|T|\phi}\sum\limits_{e\in \mathcal{E}^T_1}\int_{e} \bfu_h \cdot \bfn_e d\sigma\right)\geq 0,
    \end{align*}
    which completes the proof of the upper bound preserving property (\ref{explicit_property2}) if $c_h^{n-1}|_T\leq 1$.
    
    For the case of $ T \in \mathcal{T}_h$, $\partial T \cap \partial\Omega \neq \emptyset$, we can similarly derive (\ref{explicit_property1}) and (\ref{explicit_property2}) under the similar restriction of the time step size.
\end{proof}
\begin{theorem}\label{implicit thm}
    The solution of concentration of the transport equation (\ref{transport})-(\ref{transport_bc}) approximated by the implicit scheme (\ref{implicit transport scheme}) has the positive and upper bound preserving properties, which means
    \begin{align}
        &c_h^{n}\geq0, \qquad {\rm if} \  c_h^{n-1}\geq 0,\label{implicit_property1} \\
        &c_h^{n}\leq 1, \qquad {\rm if}\  c_h^{n-1}\leq 1.\label{implicit_property2}
    \end{align}
\end{theorem}
\begin{proof}
    For simplicity, we assume that the source/sink term is zero. Now we consider the case in a local element $T$. Similar to the definitions of (\ref{outflow}) and (\ref{inflow}), we denote the inflow and outflow interior faces or edges sets of $\partial T$ as follows:
    \begin{align*}
        &\mathcal{E}^T_1 := \{ e\in \partial T \cap \Gamma_h: \bfu_h \cdot \bfn_e \geq 0\}, \\
        &\mathcal{E}^T_2 := \{ e\in \partial T \cap \Gamma_h: \bfu_h \cdot \bfn_e < 0\}.
    \end{align*}
    Besides, we denote the  faces or edges sets of $\partial T$ on the boundaries as follows:
    \begin{align*}
        &\mathcal{E}^T_3 := \{ e\in \partial T \cap \Gamma_{out}: \bfu_h \cdot \bfn_e \geq 0\}, \\
        &\mathcal{E}^T_4 := \{ e\in \partial T \cap \Gamma_{in}: \bfu_h \cdot \bfn_e < 0\}.
    \end{align*}
    The local mass conservation on the element $T$ is described as follows:
    \begin{align}\label{implicit_conservation_4edges}
        \sum\limits_{e\in \mathcal{E}^T_1}\int_{e}\bfu_h \cdot \bfn_e d\sigma
        + \sum\limits_{e\in \mathcal{E}^T_2}\int_{e} \bfu_h \cdot \bfn_e d\sigma 
        + \sum\limits_{e\in \mathcal{E}^T_3}\int_{e}\bfu_h \cdot \bfn_e d\sigma
        + \sum\limits_{e\in \mathcal{E}^T_4}\int_{e} \bfu_h \cdot \bfn_e d\sigma= 0.
    \end{align}
      Let $q_h|_T =1$ and $q_h =0$ on other elements in (\ref{implicit transport scheme}), we obtain
    \begin{align}
        \int_T \phi \frac{c_h^n-c_h^{n-1}}{\Delta t}dx
        &+\sum\limits_{e\in \mathcal{E}^T_1}\int_{e} c_h^{n}\bfu_h \cdot \bfn_e d\sigma
        +\sum\limits_{e\in \mathcal{E}^T_2}\int_{e} c_h^{n,out}\bfu_h \cdot \bfn_e d\sigma\label{local_transpose_allterm} \\
        &+\sum\limits_{e\in \mathcal{E}^T_4}\int_e c_I \bfu_h \cdot \bfn_e d\sigma
        +\sum\limits_{e\in \mathcal{E}^T_3}\int_e c_h^{n} \bfu_h \cdot \bfn_e d\sigma := I_1 + I_2 + I_3 + I_4 + I_5 = 0, \notag
    \end{align}
    where $c_h^{n,out}$ denotes the concentration on the outflow interior faces or edges. 
    By (\ref{local_transpose_allterm}) we consider all the elements $T \in \mathcal{T}_h$ and then we get the linear system $M_hC^n_h = F_h$ for the solution of concentration in (\ref{implicit transport scheme}) where $C^n_h$ denotes the column vector of the concentration $c_h^n$ on all the elements. The term concerning $c^{n-1}_h$ in $I_1$ and $I_4$ in (\ref{local_transpose_allterm}) is assembled on the right-hand side as $F_h$, which yields that the elements in $F_h$ is always non-negative.
    
    For the coefficient matrix $M_h$ which is obviously an $L$-matrix (cf. \cite{Berman1994}), we will observe the $j$-th column of $M_h$ without loss of generality. The $j$-th diagonal element of $M_h$ is obtained from the term concerning $c^n_h$ in $I_1$, $I_2$ and $I_5$ in (\ref{local_transpose_allterm}), which means the positive coefficient on the $j$-th element. Besides, there might be some negative terms in the column vector if $I_2 \neq 0$, whose absolute value is equal to part of the $j$-th diagonal element of $M_h$ while the sign of each value is opposite. This indicates that the coefficient matrix $M_h$ is a column strictly diagonally dominant matrix. By the Gershgorin circle theorem (cf. \cite{Salas1999}), all the real part of eigenvalues of $M_h$ are positive, which directly yields that $M_h$ is a $M$-matrix (cf. \cite{Berman1994}). By the property of $M$-matrix that the inverse of $M_h$ is non-negative, we can conclude that the solution $C^n_h$ of the above linear system is nonnegative, which completes the proof of the positive preserving property (\ref{implicit_property1}) if $c_h^{n-1} \geq 0$. 
    
    Besides, by (\ref{implicit_conservation_4edges}) and (\ref{local_transpose_allterm}), we directly have
    \begin{align*}
        \int_T \phi \frac{(1-c_h^n)-(1-c_h^{n-1})}{\Delta t}dx
        &+\sum\limits_{e\in \mathcal{E}^T_1}\int_{e} (1-c_h^{n})\bfu_h \cdot \bfn_e d\sigma
        +\sum\limits_{e\in \mathcal{E}^T_2}\int_{e} (1-c_h^{n,out})\bfu_h \cdot \bfn_e d\sigma \\
        &+\sum\limits_{e\in \mathcal{E}^T_4}\int_e (1-c_I) \bfu_h \cdot \bfn_e d\sigma
        +\sum\limits_{e\in \mathcal{E}^T_3}\int_e (1-c_h^{n}) \bfu_h \cdot \bfn_e d\sigma \\
       & := J_1 + J_2 + J_3 + J_4 + J_5 = 0.
    \end{align*}
    We can also create the linear system for the solution of $1-c^n_h$ and the analysis for the upper bound preserving property (\ref{implicit_property2}) can be similarly derived if $c_h^{n-1}\leq 1$.  
\end{proof}

We note that in the explicit scheme (\ref{explicit transport scheme}), the time step size $\Delta t$ needs to satisfy the restriction (\ref{timestep}) on all the elements $T \in \mathcal{T}_h$. For the implicit scheme (\ref{implicit transport scheme}), one needs to solve a linear system to get the solution. However,  no restriction of $\Delta t$ occurs in the implicit scheme (\ref{implicit transport scheme}).

%%%%%%%%%%%%%%%%%%%%%%%%%%%%%%%%%%%%%%%%%%%%%%%%%%%%%%%
%----------------section 5-----------------------------
\section{Numerical experiments}\label{section5}
%%%%%%%%%%%%%%%%%%%%%%%%%%%%%%%%%%%%%%%%%%%%%%%%%%%%%%%
In this section, we present several numerical examples in two dimensions to demonstrate the performance of our proposed EPG method. We consider a coupled problem of the Darcy flow and the species transport. The Darcy flow is solved by the CG and EPG methods while the transport equation is solved by the DG method with a piecewise constant approximation. The implicit discrete scheme (\ref{implicit transport scheme}) with uniform time step size is used for the simulation of the species transport. Since the scheme (\ref{implicit transport scheme}) is implicit, the time step size does not need to be small enough. In the following examples, the CG method is implemented for piecewise linear (P1-CG), piecewise quadratic (P2-CG), and piecewise cubic (P3-CG) continuous finite element spaces, and the EPG method is implemented for piecewise linear (P1-EPG), piecewise quadratic (P2-EPG), and piecewise cubic (P3-EPG) continuous finite element spaces coupled with the piecewise bubble functions and the piecewise constant functions as the trial and test spaces. We provide the results of convergence of the methods if the example has an exact solution. Moreover, for the transport problem in each example, the porosity $\phi$ is 0.2, the inflow concentration $c_I$ is one and the initial concentration $c_0$ is zero. Besides, the total number of iterative steps for the implicit scheme (\ref{implicit transport scheme}) is chosen as 100.
%-------------------------example1---------------------------
\begin{example}\label{example1}
    In this example, we firstly consider the Darcy flow with an exact solution in the unit square domain $\Omega = [0,1]^2$ and the conductivity $\bfK$ is a diagonal tensor with its entry being 1 in the entire domain. The exact pressure is denoted by $p=  (1-x)  y  (1-y)\cos{x}$ and the Dirichlet condition is imposed on the boundary. 
    
    The solutions of the pressure and the velocity are shown in Figure \ref{ex1_fig1}. Here the problem is simulated on a triangular mesh with 512 elements. Furthermore, the convergence results of the pressure and the velocity are displayed in Figure \ref{ex1_fig2} to verify the results in Theorem \ref{errorp} and Theorem \ref{erroru}. Here we use a log-log scale with $x$-axis meaning the number of degrees of freedom and $y$-axis meaning the relative error. The optimal convergence of the CG and EPG methods can be both observed from Figure \ref{ex1_fig2}.
    
\begin{figure}[htbp]
    \includegraphics[width=3.5cm,height=3cm]{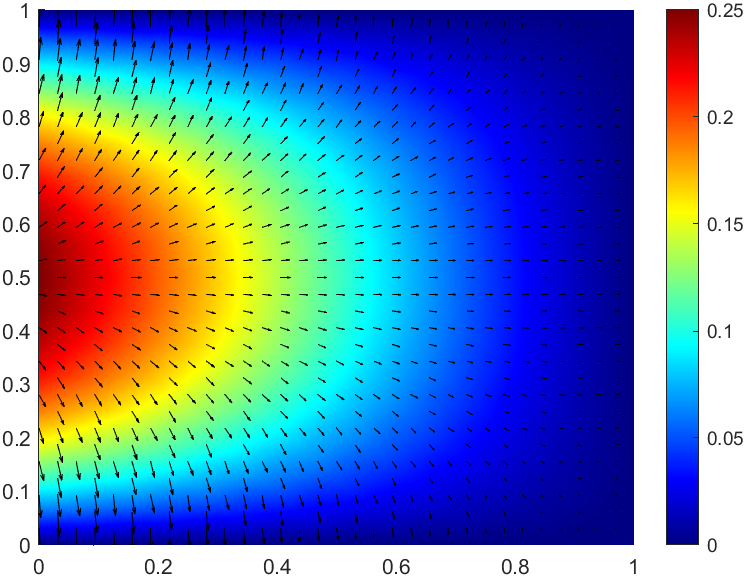}\qquad
    \includegraphics[width=3.5cm,height=3cm]{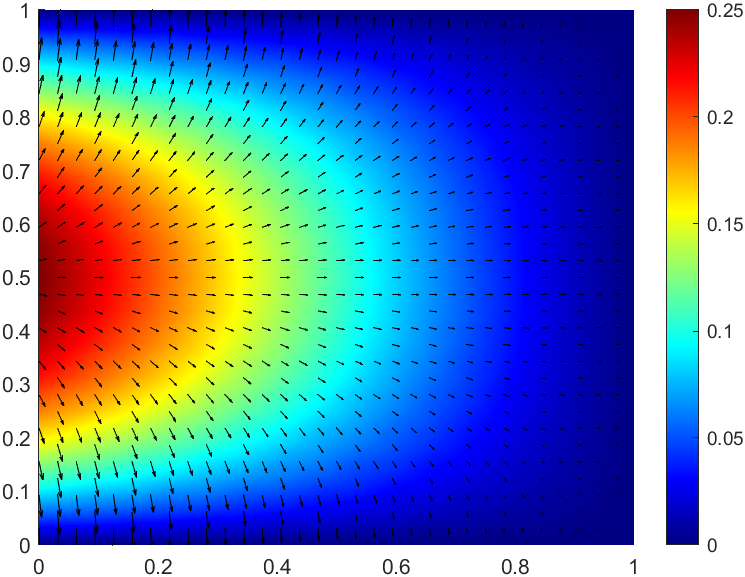}\qquad
    \includegraphics[width=3.5cm,height=3cm]{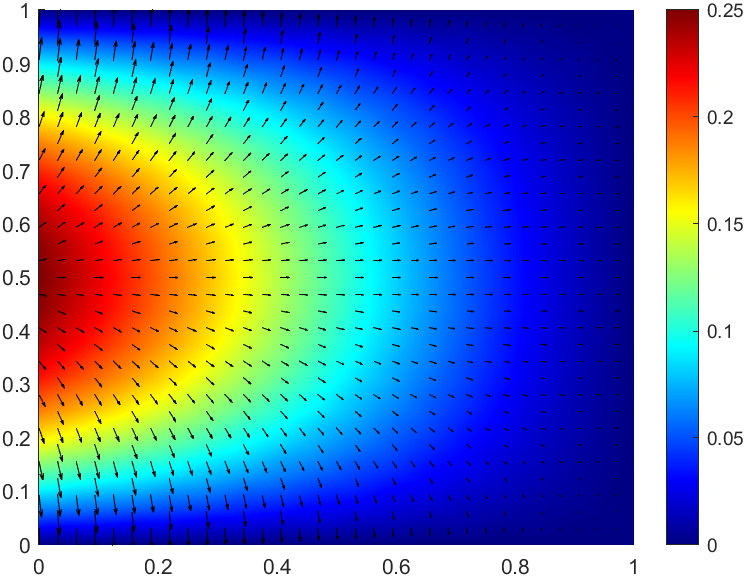}\\
    \includegraphics[width=3.5cm,height=3cm]{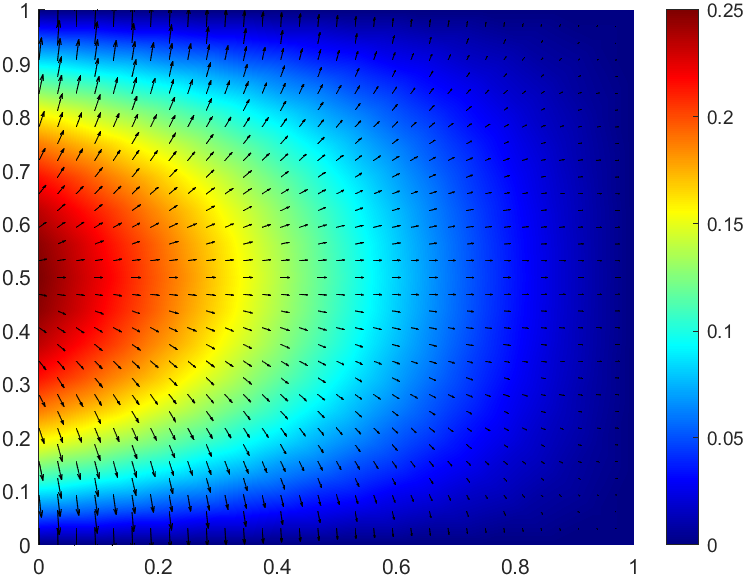}\qquad
    \includegraphics[width=3.5cm,height=3cm]{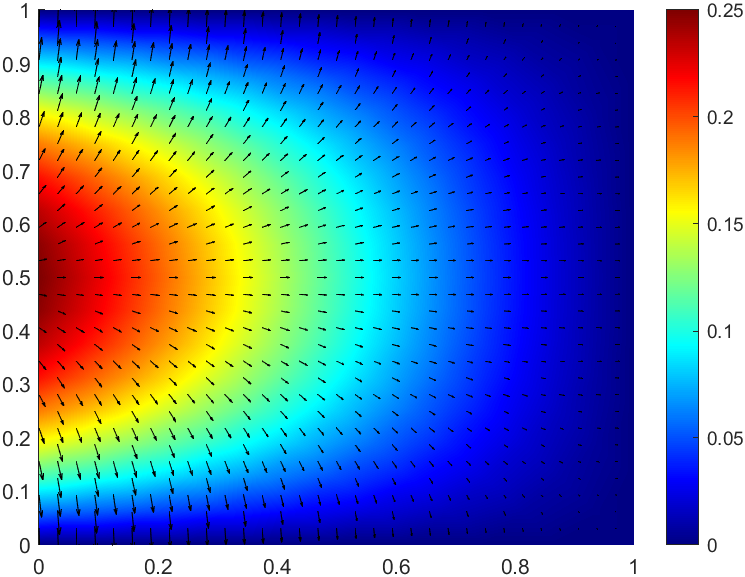}\qquad
    \includegraphics[width=3.5cm,height=3cm]{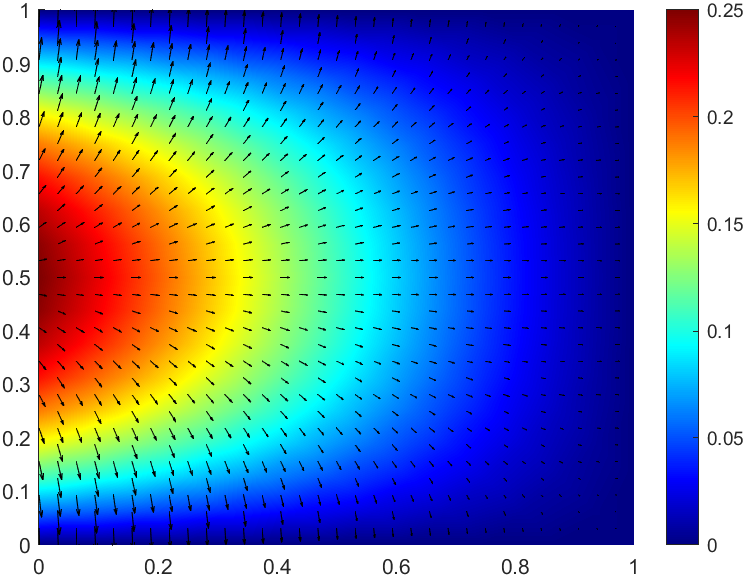}
    \caption{\footnotesize (Example \ref{example1})  The solutions of pressure and Darcy velocity. Top-left: P1-CG. Top-middle: P2-CG. Top-right: P3-CG. Bottom-left: P1-EPG. Bottom-middle: P2-EPG. Bottom-right: P3-EPG.}
    \label{ex1_fig1}
\end{figure}

\begin{figure}[htbp]
    \includegraphics[width=3.5cm,height=3cm]{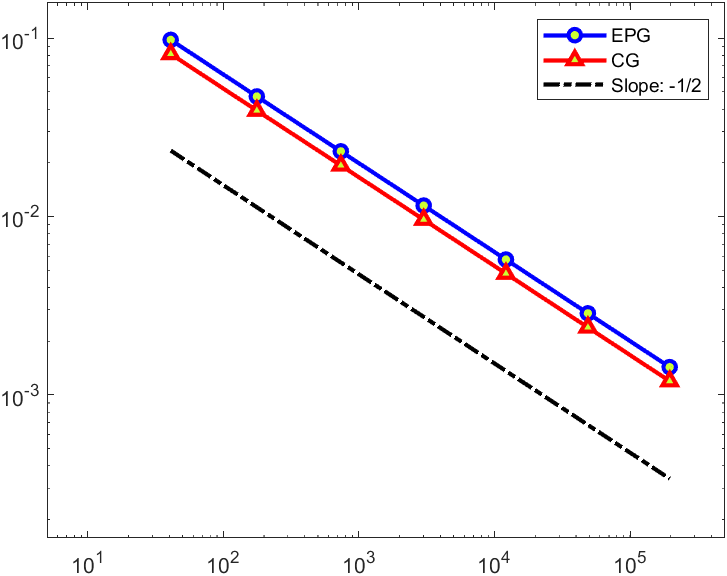}\qquad
    \includegraphics[width=3.5cm,height=3cm]{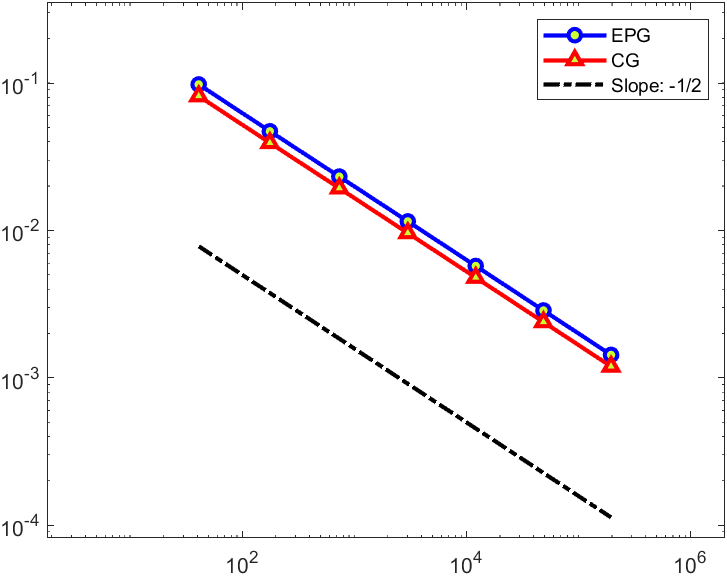}\qquad
    \includegraphics[width=3.5cm,height=3cm]{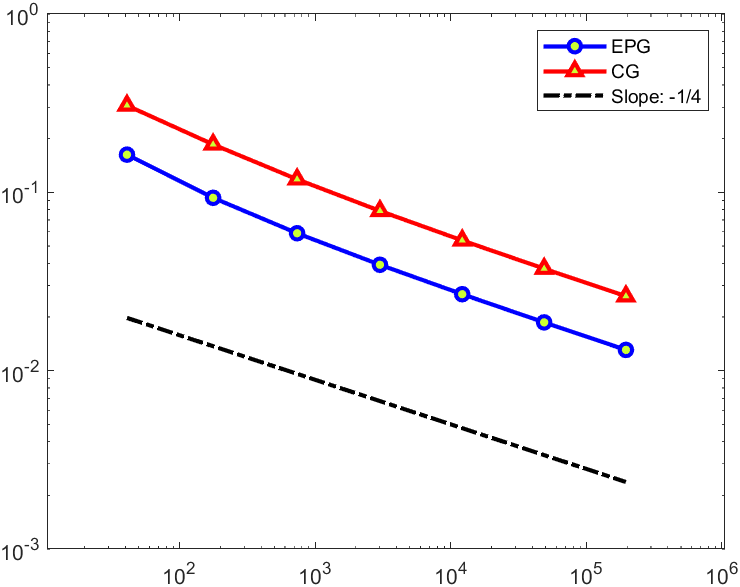}\\
    \includegraphics[width=3.5cm,height=3cm]{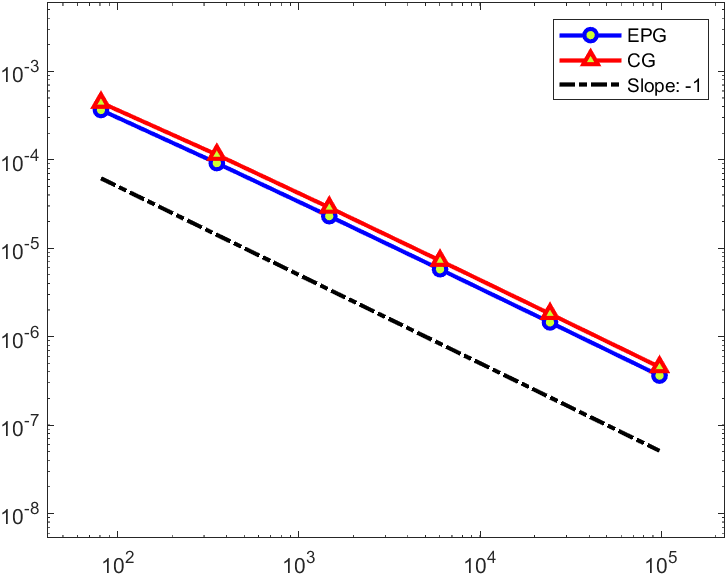}\qquad
    \includegraphics[width=3.5cm,height=3cm]{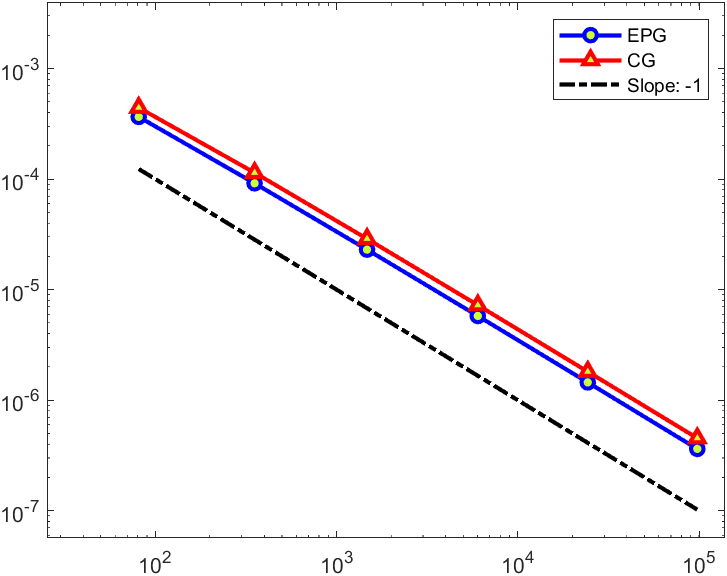}\qquad
    \includegraphics[width=3.5cm,height=3cm]{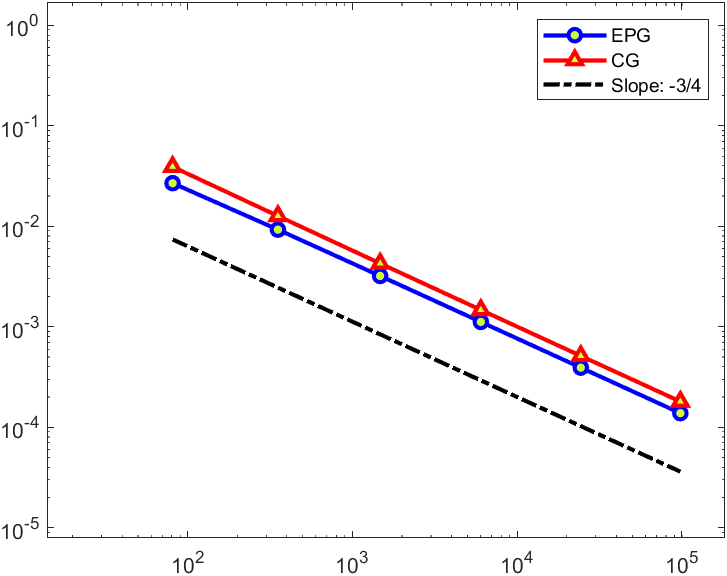}\\
    \includegraphics[width=3.5cm,height=3cm]{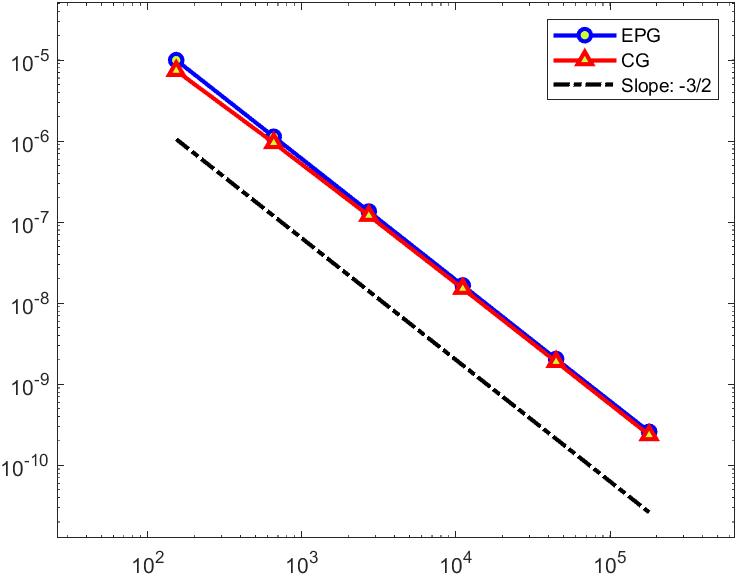}\qquad
    \includegraphics[width=3.5cm,height=3cm]{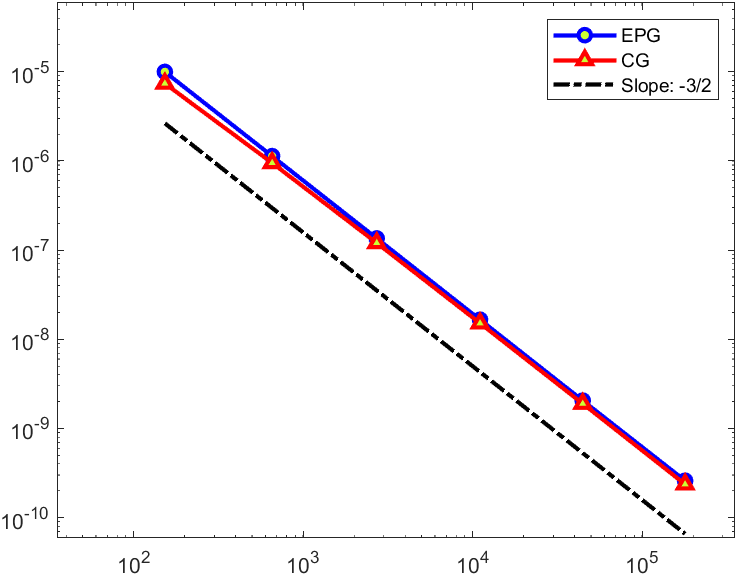}\qquad
    \includegraphics[width=3.5cm,height=3cm]{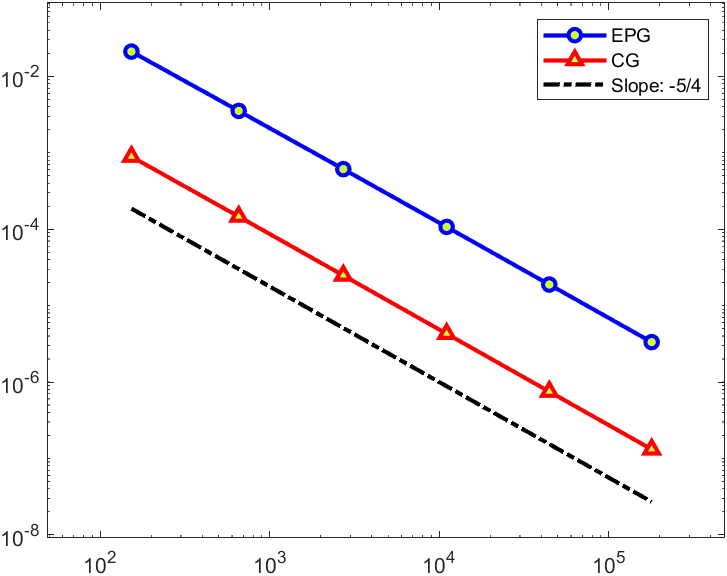}
    \caption{\footnotesize (Example \ref{example1})  Relative errors. Top-bottom: P1, P2 and P3. Left: Energy error of pressure. Middle: $L^2$ error of velocity. Right: $L^2$ error of the normal component of velocity on the trace.}
    \label{ex1_fig2}
\end{figure}
    We further compute the local mass conservation residual in Figure \ref{ex1_fig3}, which is defined as follows:
    \begin{align}\label{local_equal}
        R_{loc} :=\int_{\partial T} \bfu_h \cdot \bfn d\sigma - \int_{T} fdx.
    \end{align}
    Here the number of elements on the triangluar mesh is 32768. From Figure \ref{ex1_fig3}, we can observe the local mass conservation of the EPG method in this example is in the order of $10^{-17}$, which is close to machine precision. However, the local mass conservation residual for the CG method is relatively large in the order of $10^{-5}$. Then we use the respective velocities from the CG and EPG methods to compute the concentration in the transport equation. The time step size is chosen to be 0.05. Figure \ref{ex1_fig4} shows the simulations of the concentration at the time 4.5 based on the velocity from the CG method while Figure \ref{ex1_fig7} is presented to show the simulations of the concentration at the time 0.2, 0.4 and 0.8 based on the velocity from the EPG method. Since the simulating results for the concentration based on the velocity from P1-EPG, P2-EPG and P3-EPG are almost the same, we also only show the simulating results for the concentration based on the velocity from P3-EPG in the following examples. The maximum concentrations of the transport equation based on the velocity from the CG and EPG methods are shown in Figure \ref{ex1_fig8}. From the above Figures, we can conclude that the concentration will not overshoot if the velocity based on the EPG method is used for the solution of transport equation. However, the concentration will overshoot if the velocity based on the CG method is used.

\begin{figure}[htbp]
    \includegraphics[width=3.5cm,height=3cm]{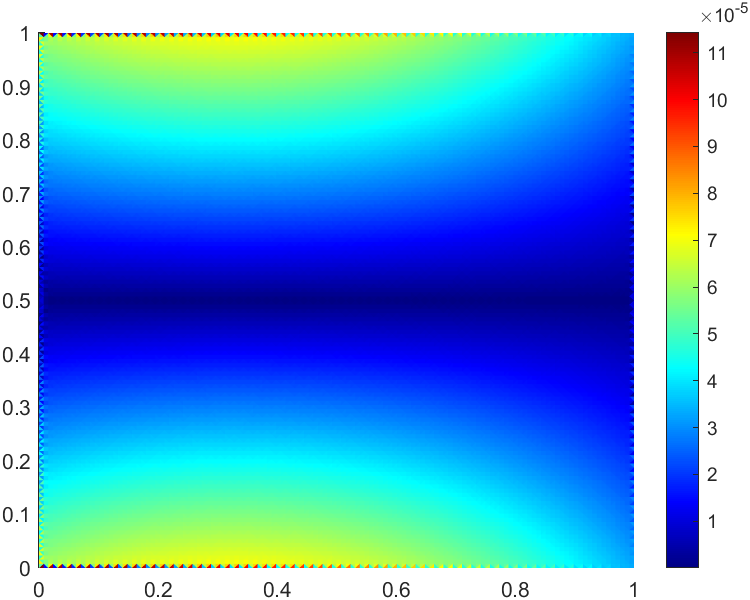}\qquad
    \includegraphics[width=3.5cm,height=3cm]{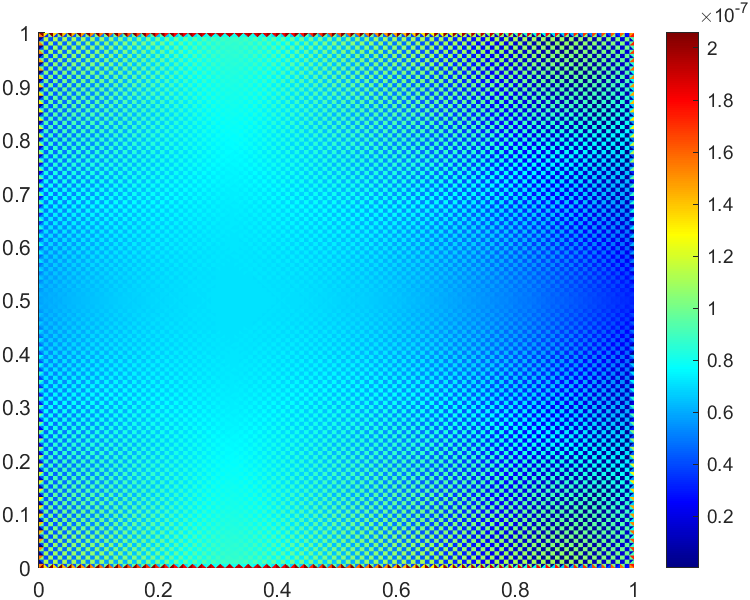}\qquad
    \includegraphics[width=3.5cm,height=3cm]{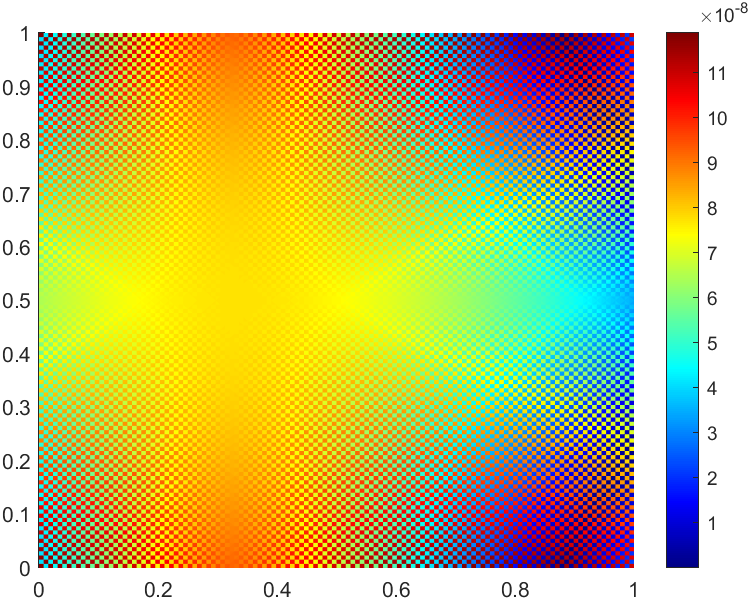}\\
    \includegraphics[width=3.5cm,height=3cm]{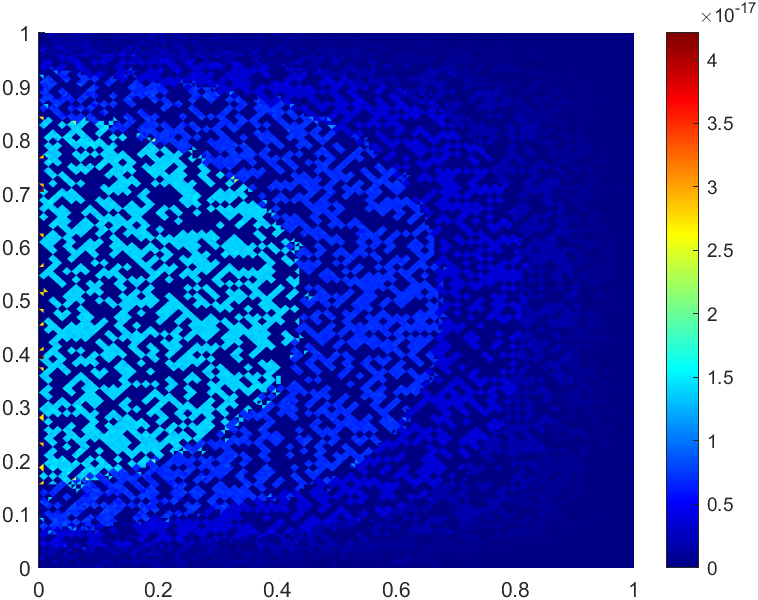}\qquad
    \includegraphics[width=3.5cm,height=3cm]{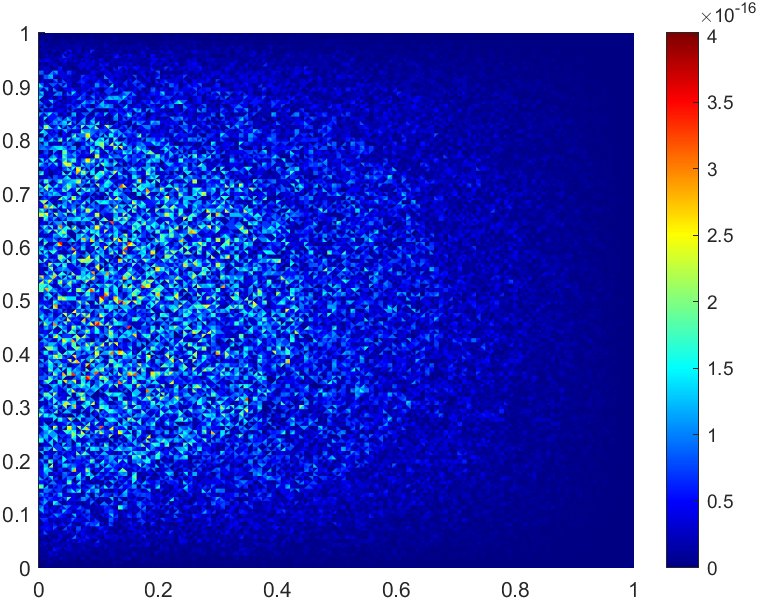}\qquad
    \includegraphics[width=3.5cm,height=3cm]{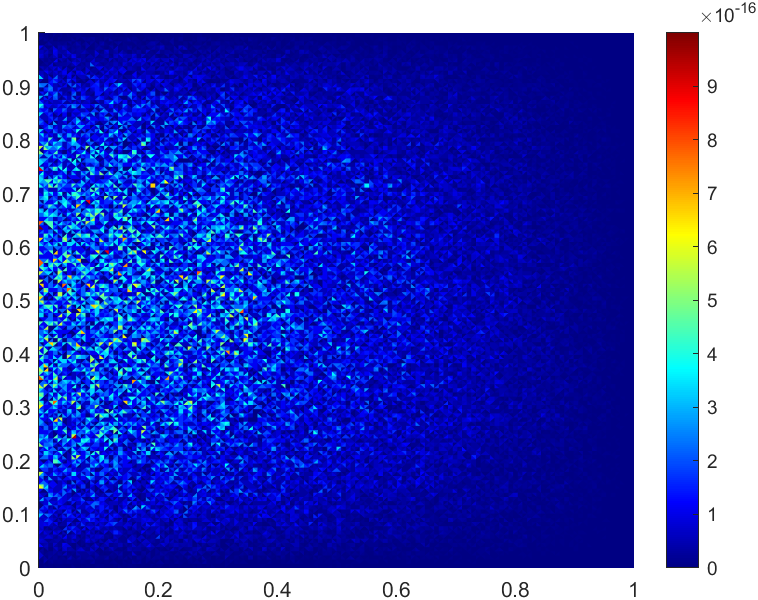}
    \caption{\footnotesize (Example \ref{example1})  Local mass conservation residual. Top-left: P1-CG. Top-middle: P2-CG. Top-right: P3-CG. Bottom-left: P1-EPG. Bottom-middle: P2-EPG. Bottom-right: P3-EPG.}
    \label{ex1_fig3}
\end{figure}

\begin{figure}[htbp]
    \includegraphics[width=3.5cm,height=3cm]{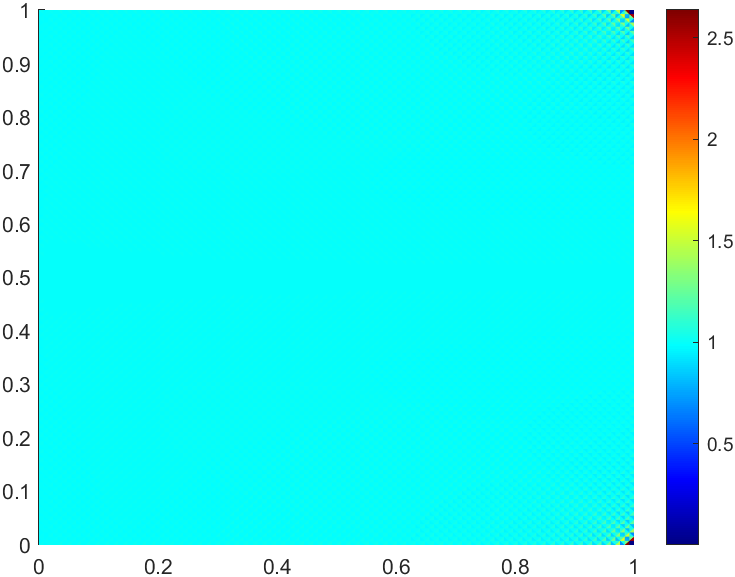}\qquad
    \includegraphics[width=3.5cm,height=3cm]{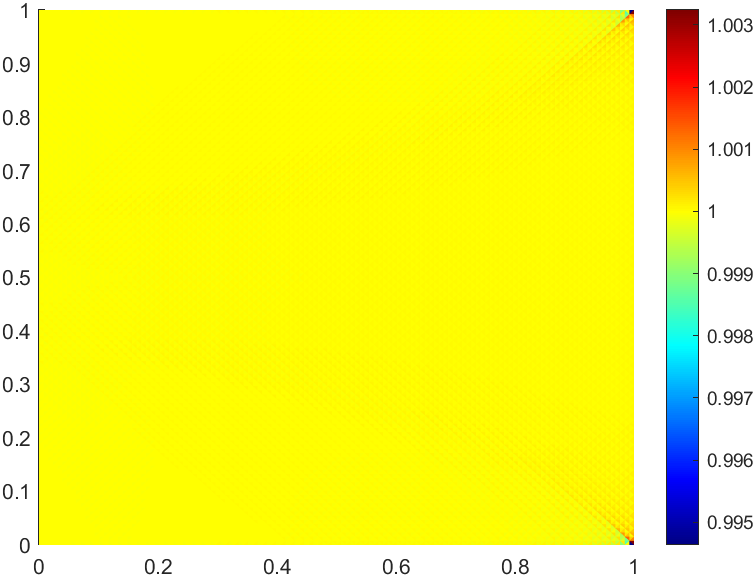}\qquad
    \includegraphics[width=3.5cm,height=3cm]{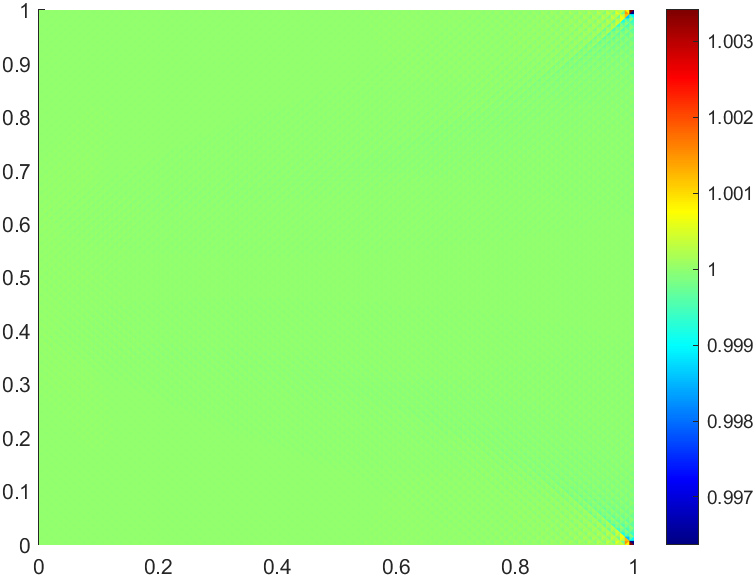}
    \caption{\footnotesize (Example \ref{example1})  Simulations of concentration at time 4.5 based on the velocity from the CG method. Left: P1-CG. Middle: P2-CG. Right: P3-CG.}
    \label{ex1_fig4}
\end{figure}

\begin{figure}[htbp]
    \includegraphics[width=3.5cm,height=3cm]{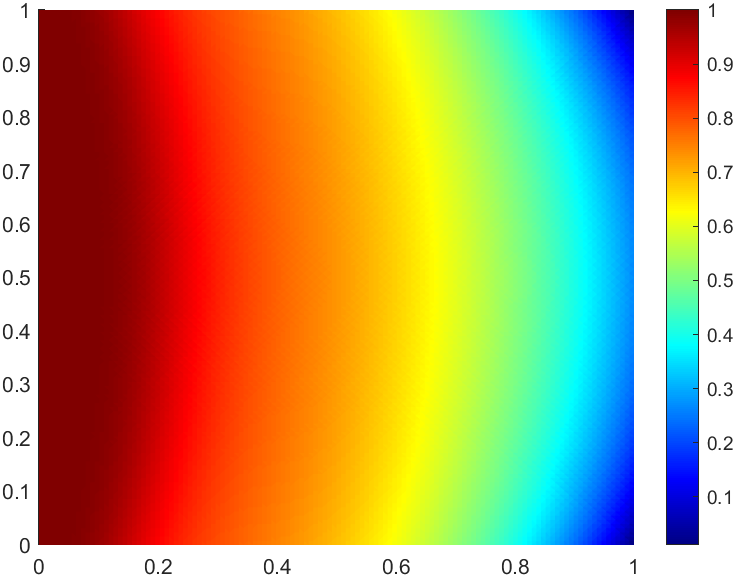}\qquad
    \includegraphics[width=3.5cm,height=3cm]{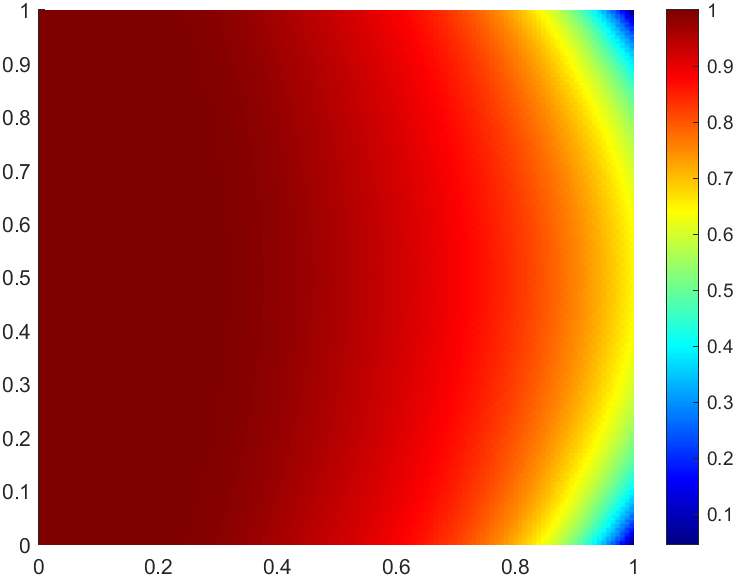}\qquad
    \includegraphics[width=3.5cm,height=3cm]{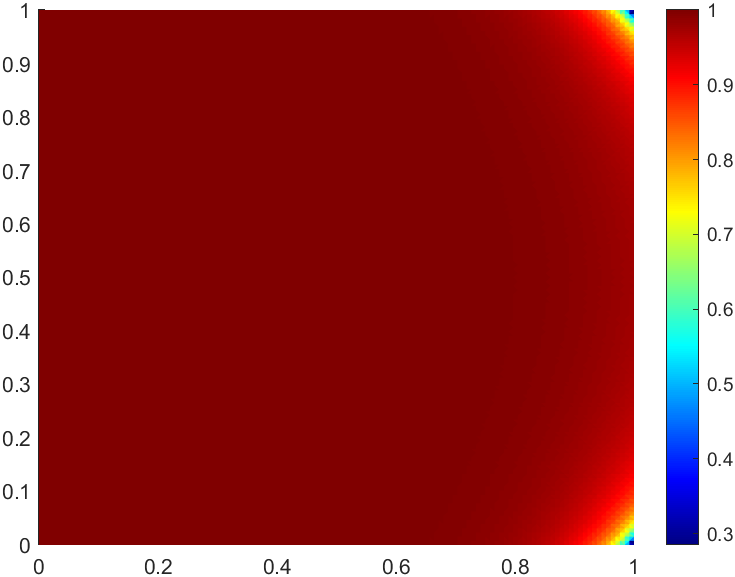}
    \caption{\footnotesize (Example \ref{example1})  Simulations of concentration based on the velocity from P3-EPG. Left-right: Simulations at time 0.2, 0.4 and 0.8.}
    \label{ex1_fig7}
\end{figure}

\begin{figure}[htbp]
    \includegraphics[width=3.5cm,height=3cm]{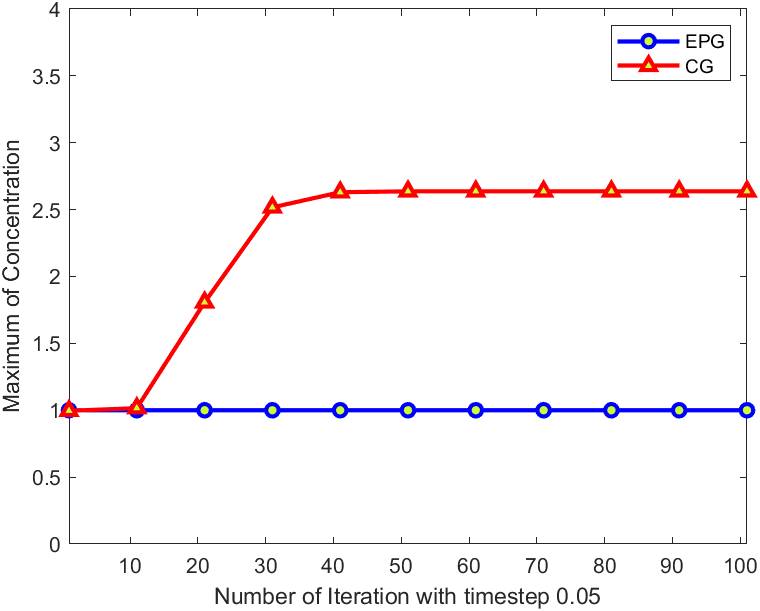}\qquad
    \includegraphics[width=3.5cm,height=3cm]{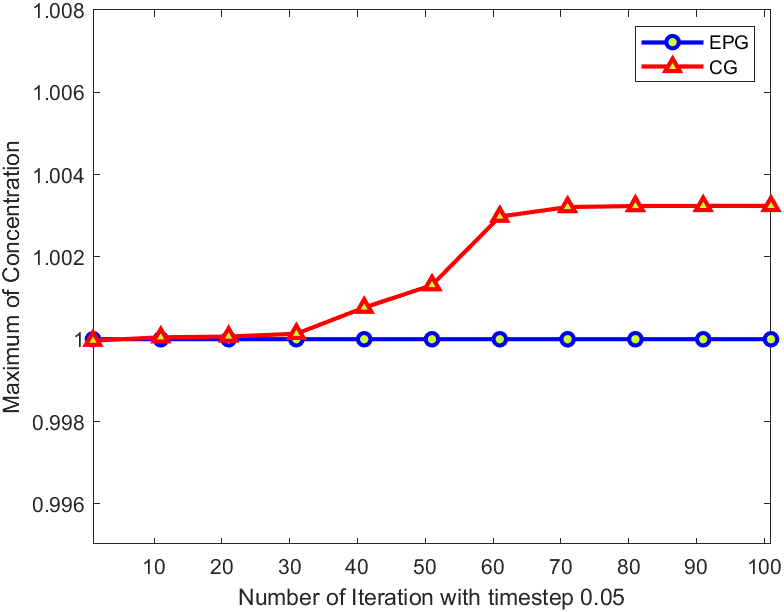}\qquad
    \includegraphics[width=3.5cm,height=3cm]{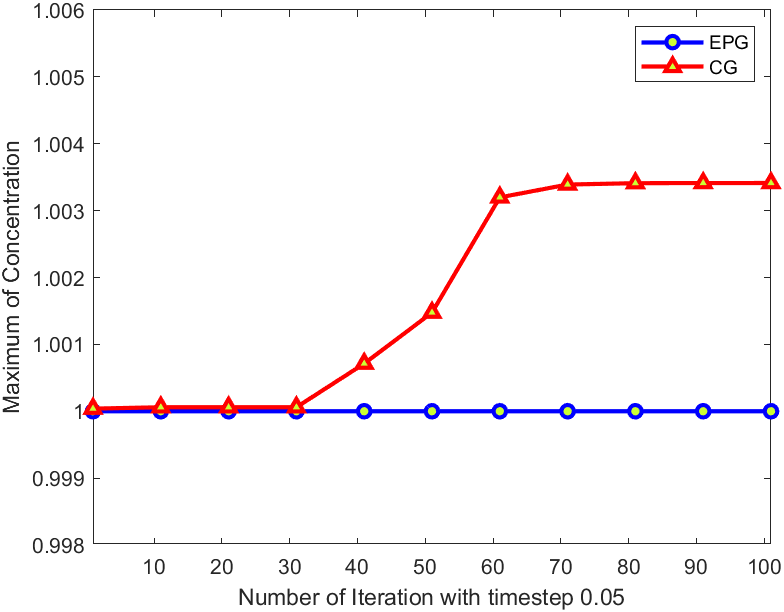}
    \caption{\footnotesize (Example \ref{example1})  Maximum concentrations based on the velocity from the CG and EPG methods. Left-right: P1, P2 and P3.}
    \label{ex1_fig8}
\end{figure}
\end{example}

%----------------------example3-------------------------
\begin{example}\label{example3}
    Next we consider a ten-shaped domain $\Omega$, which consists of five unit square subdomains and the central point is ($\frac{3}{2},\frac{3}{2}$). The central subdomain $\Omega_c = [\frac{5}{4},\frac{7}{4}]^2$. The conductivity $\bfK$ is still a diagonal tensor with its entry being $10^{-2}$ in $\Omega_c$ and 1 elsewhere. The boundary conditions for the Darcy flow are imposed as follows:
    \begin{align*}
        &p = 1, \quad {\rm on}\ \{0\} \times (1,2), \\
        &p = 0, \quad {\rm on}\ \{3\} \times (1,2),\quad (1,2) \times \{0\},\quad (1,2) \times \{3\}, \\
        &\bfu \cdot \bfn = 0, \qquad {\rm elsewhere}.
    \end{align*}

The pressure and the velocity are shown in Figure \ref{ex3_fig1} with 640 elements. Figure \ref{ex3_fig2} describing the local mass conservation indicates that the residual (\ref{local_equal}) for the CG method is larger than that for the EPG method beside $\Omega_c$. Here we test the problem on a triangular mesh with 40960 elements. Moreover, we numerically compute the concentration. The time step size is chosen to be 0.03. It is clear that the phenomenon of overshooting based on velocity from the CG method is still obvious while the concentration will not overshoot based on velocity from the EPG method in this example, which are displayed in Figures \ref{ex3_fig3} - \ref{ex3_fig7}. By comparison, it can be concluded that the local mass-conservation preserving property of the EPG method is really important.
    
\begin{figure}[htbp]
    \includegraphics[width=3.5cm,height=3cm]{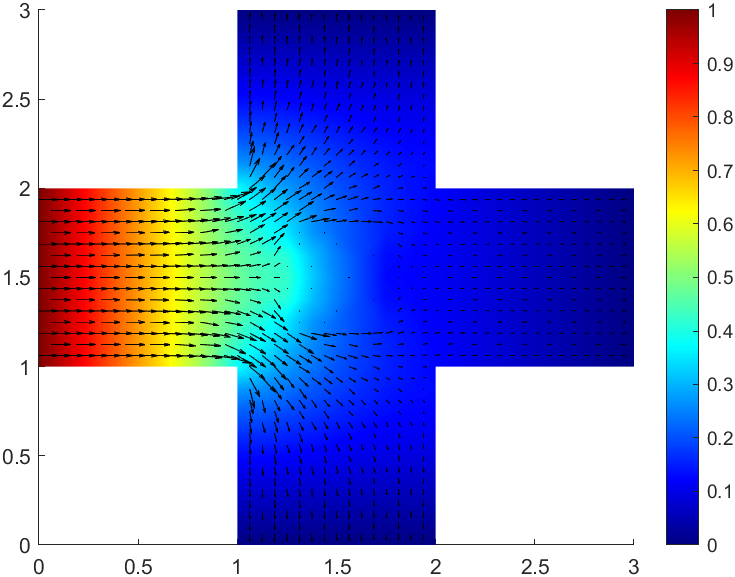}\qquad
    \includegraphics[width=3.5cm,height=3cm]{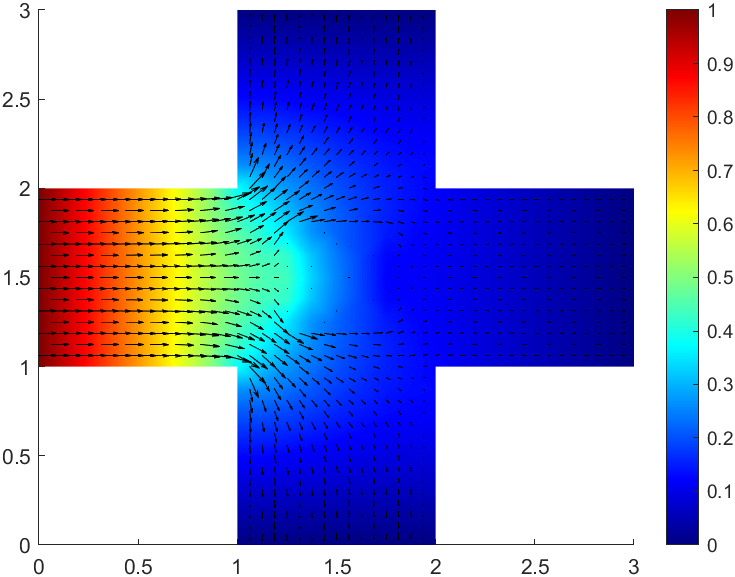}\qquad
    \includegraphics[width=3.5cm,height=3cm]{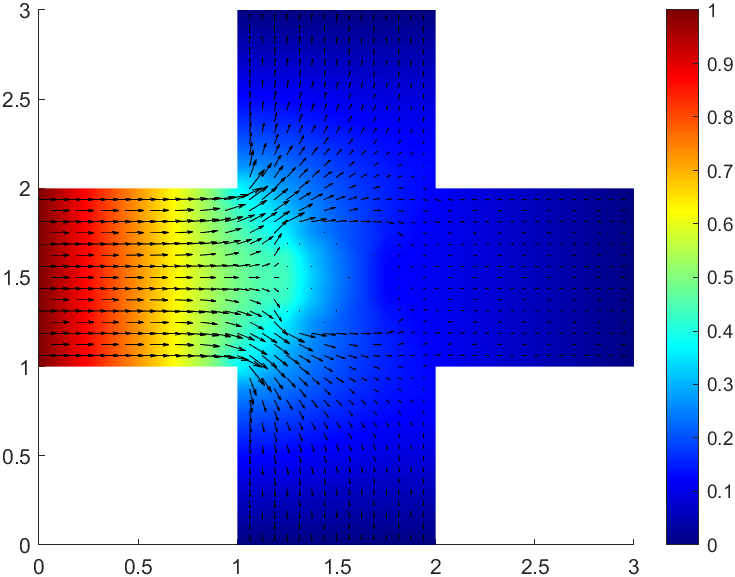}\\
    \includegraphics[width=3.5cm,height=3cm]{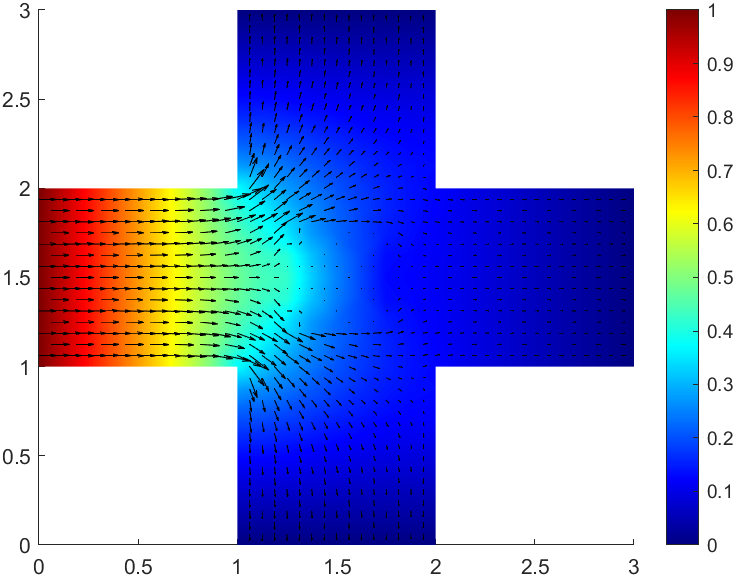}\qquad
    \includegraphics[width=3.5cm,height=3cm]{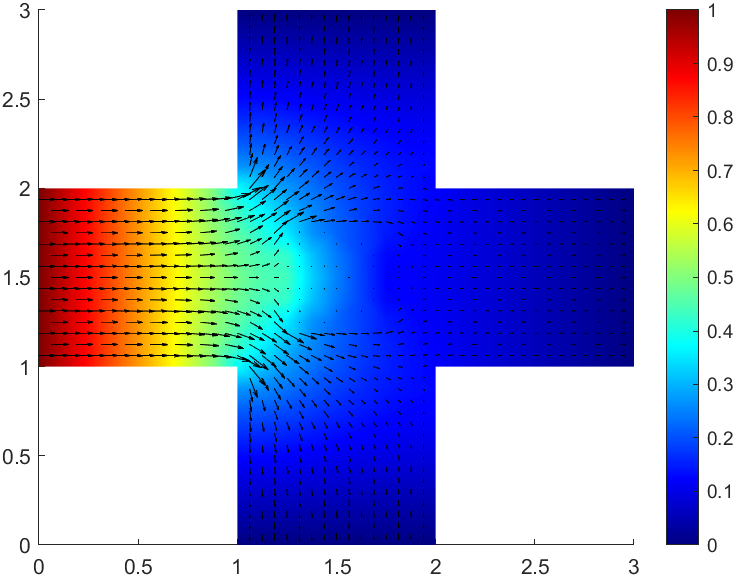}\qquad
    \includegraphics[width=3.5cm,height=3cm]{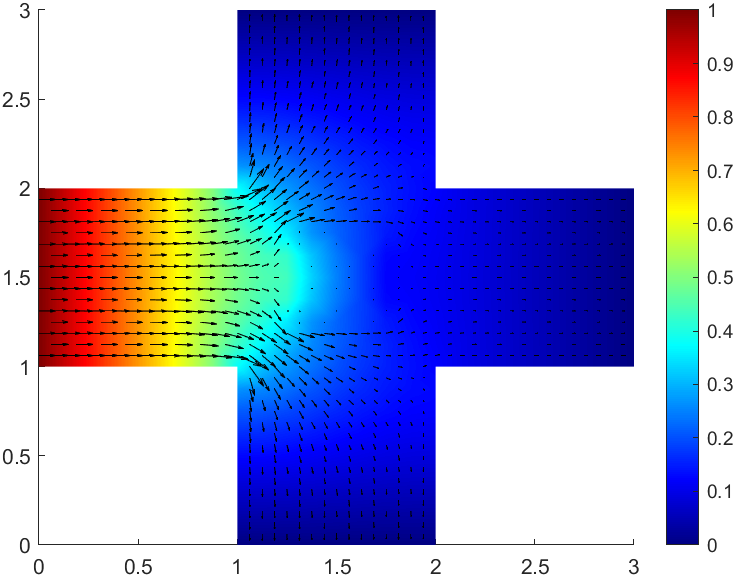}
    \caption{\footnotesize (Example \ref{example3})  The solutions of pressure and Darcy velocity. Top-left: P1-CG. Top-middle: P2-CG. Top-right: P3-CG. Bottom-left: P1-EPG. Bottom-middle: P2-EPG. Bottom-right: P3-EPG.}
    \label{ex3_fig1}
\end{figure}

\begin{figure}[htbp]
    \includegraphics[width=3.5cm,height=3cm]{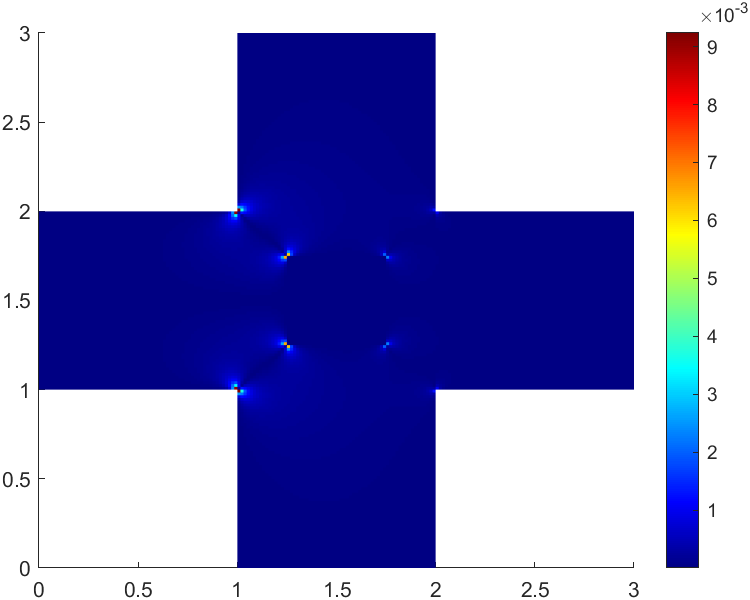}\qquad
    \includegraphics[width=3.5cm,height=3cm]{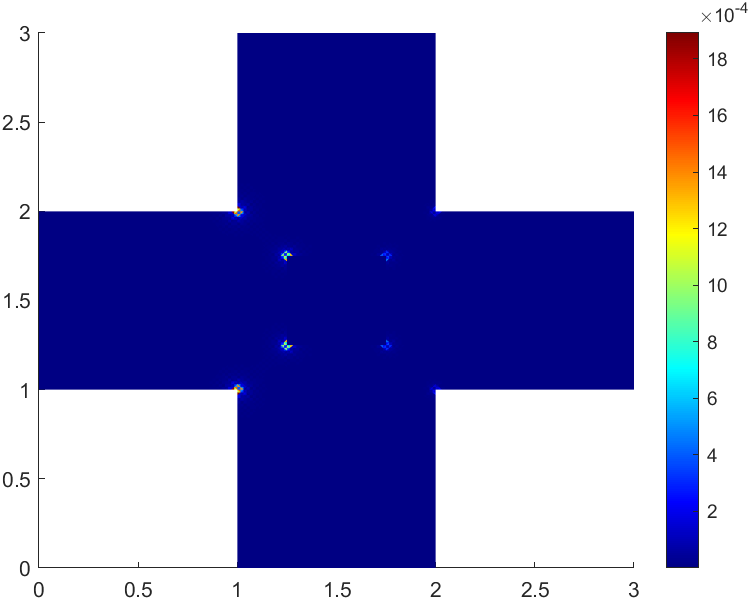}\qquad
    \includegraphics[width=3.5cm,height=3cm]{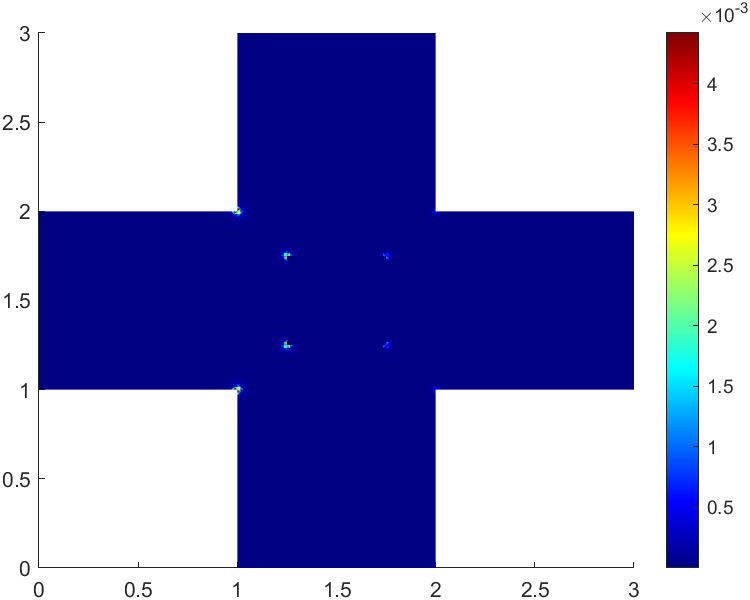}\\
    \includegraphics[width=3.5cm,height=3cm]{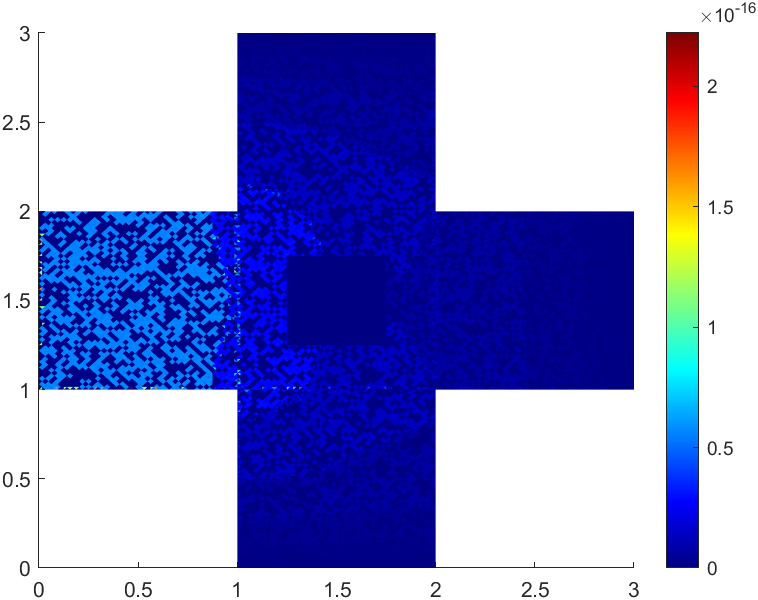}\qquad
    \includegraphics[width=3.5cm,height=3cm]{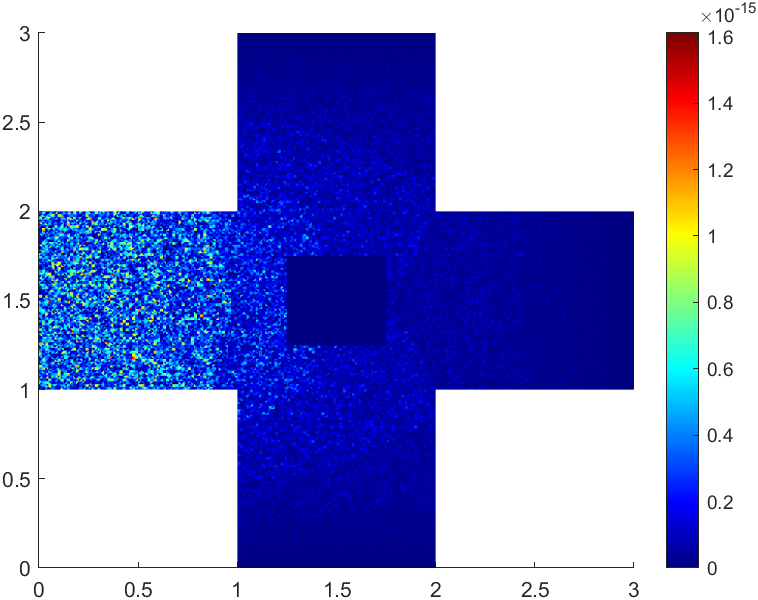}\qquad
    \includegraphics[width=3.5cm,height=3cm]{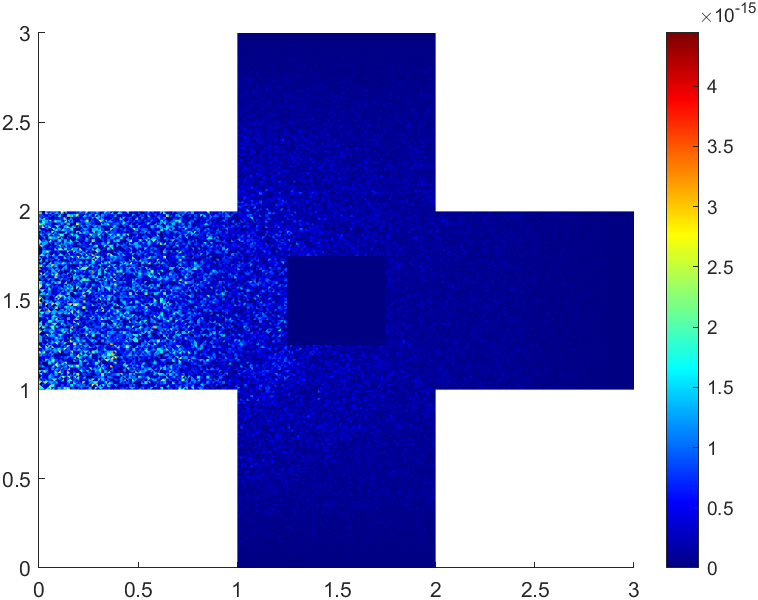}
    \caption{\footnotesize (Example \ref{example3})  Local mass conservation residual. Top-left: P1-CG. Top-middle: P2-CG. Top-right: P3-CG. Bottom-left: P1-EPG. Bottom-middle: P2-EPG. Bottom-right: P3-EPG.}
    \label{ex3_fig2}
\end{figure}

\begin{figure}[htbp]
    \includegraphics[width=3.5cm,height=3cm]{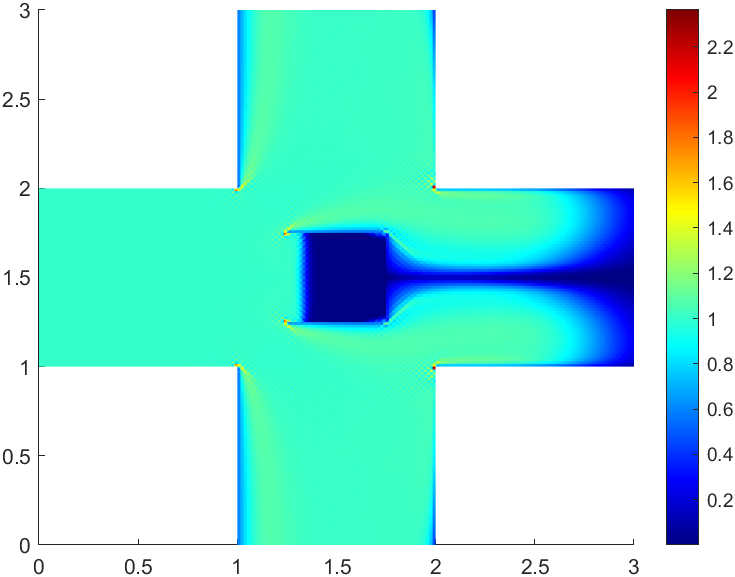}\qquad
    \includegraphics[width=3.5cm,height=3cm]{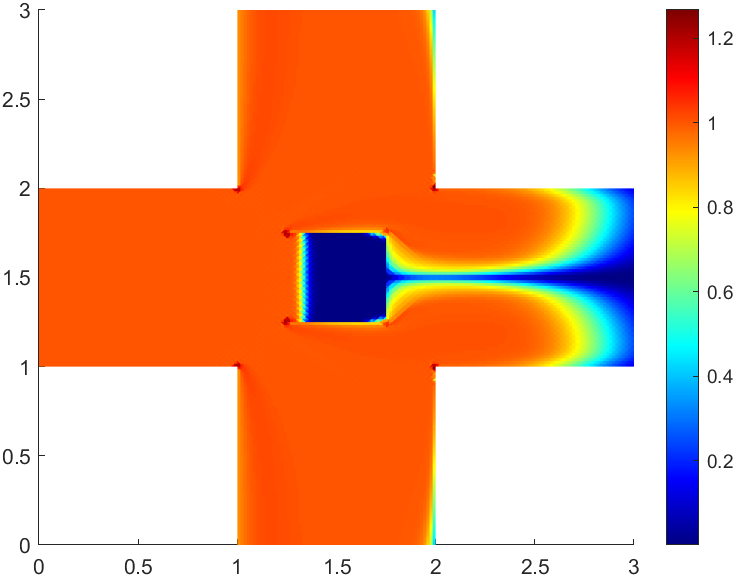}\qquad
    \includegraphics[width=3.5cm,height=3cm]{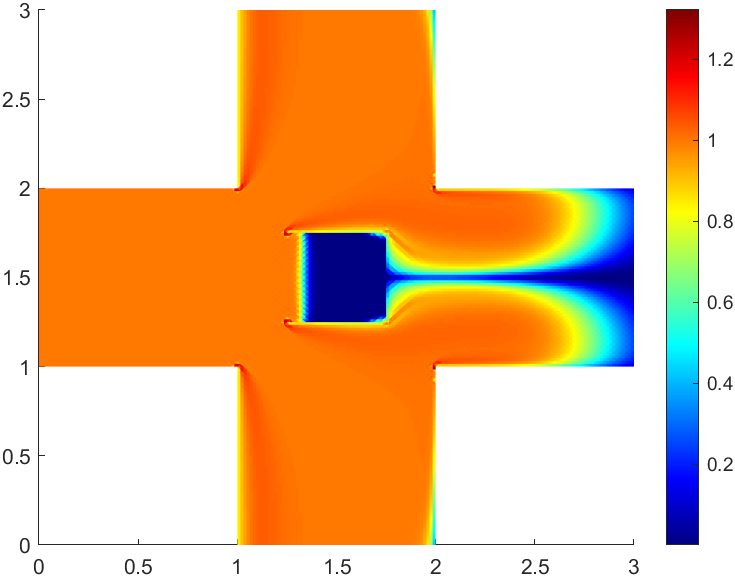}
    \caption{\footnotesize (Example \ref{example3})  Simulations of concentration at time 2.7 based on the velocity from the CG method. Left: P1-CG. Middle: P2-CG. Right: P3-CG.}
    \label{ex3_fig3}
\end{figure}

\begin{figure}[htbp]
    \includegraphics[width=3.5cm,height=3cm]{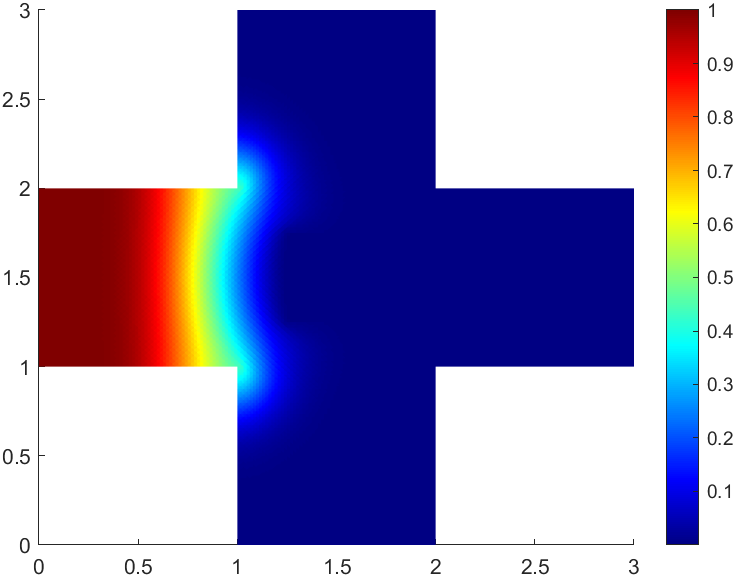}\qquad
    \includegraphics[width=3.5cm,height=3cm]{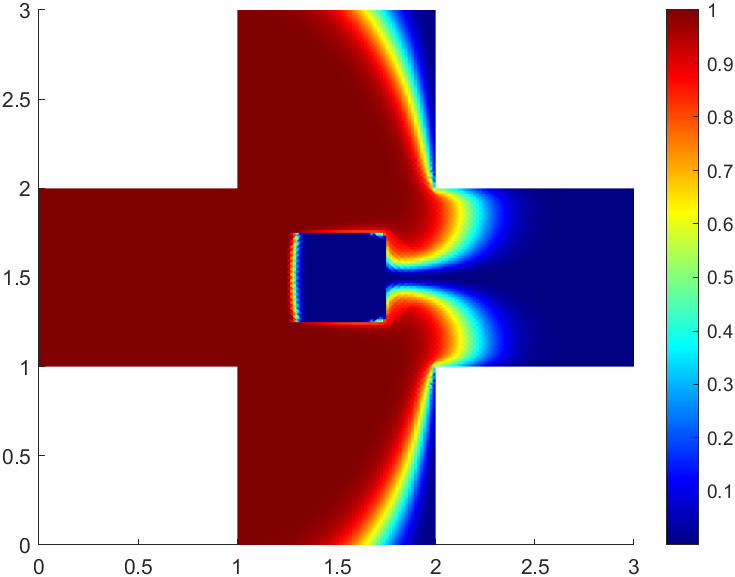}\qquad
    \includegraphics[width=3.5cm,height=3cm]{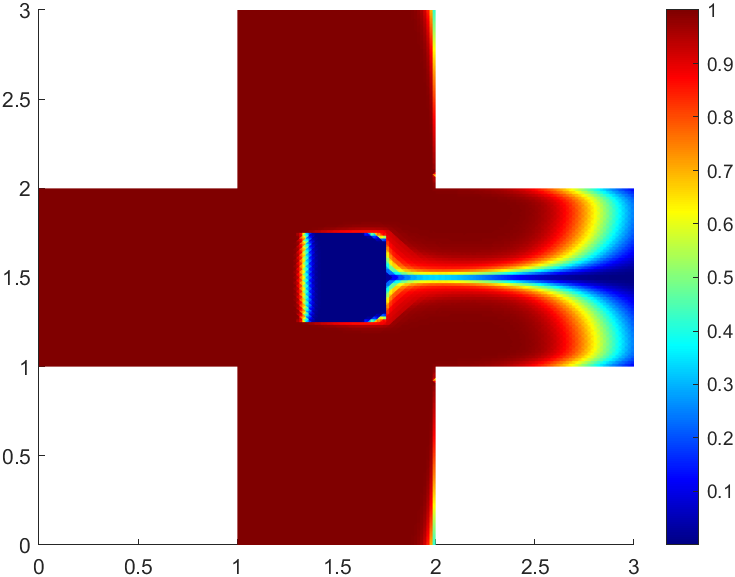}
    \caption{\footnotesize (Example \ref{example3})  Simulations of concentration based on the velocity from P3-EPG. Left-right: Simulations at time 0.3, 1.5 and 2.7.}
    \label{ex3_fig6}
\end{figure}

\begin{figure}[htbp]
    \includegraphics[width=3.5cm,height=3cm]{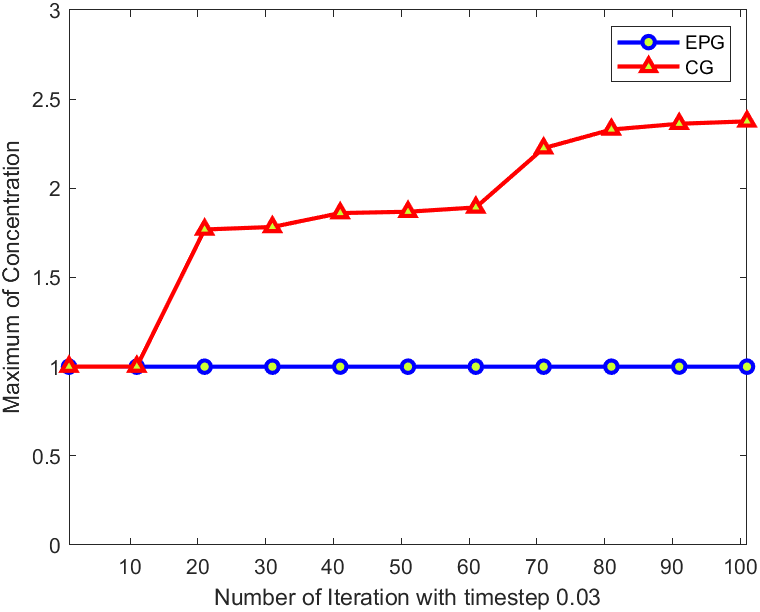}\qquad
    \includegraphics[width=3.5cm,height=3cm]{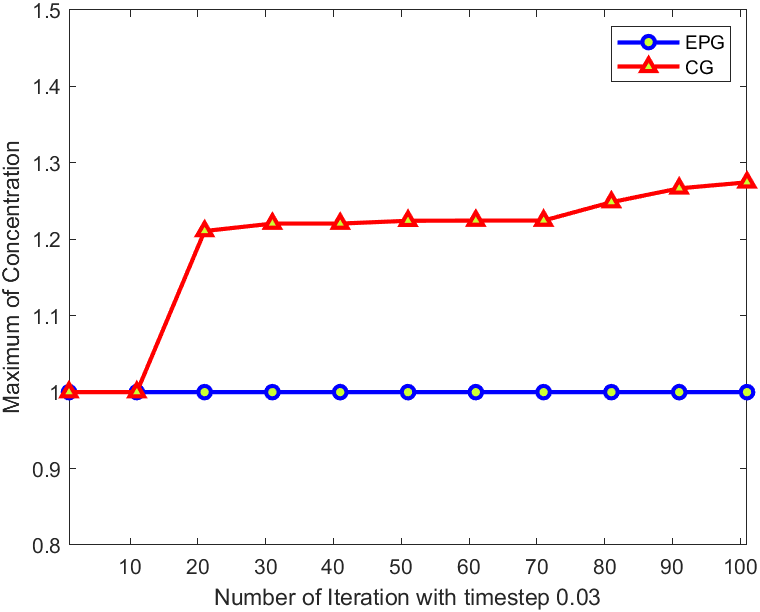}\qquad
    \includegraphics[width=3.5cm,height=3cm]{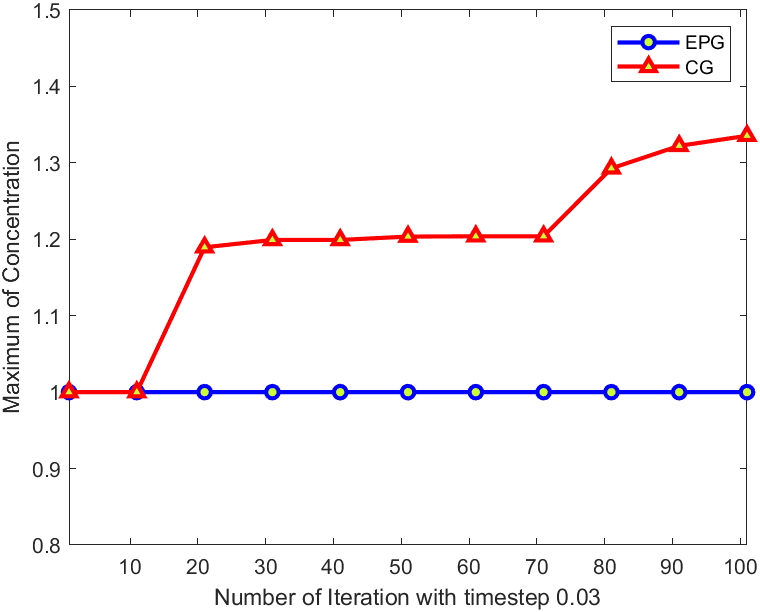}
    \caption{\footnotesize (Example \ref{example3})  Maximum concentrations based on the velocity from the CG and EPG methods. Left-right: P1, P2 and P3.}
    \label{ex3_fig7}
\end{figure}
\end{example}

%----------------------example4-------------------------
\begin{example}\label{example4}
    In the final example, we consider the L-shaped domain $\Omega = [0,2]^2\setminus [0,1]^2$. There exist two small square subdomains $\Omega_{c_1}=[\frac{1}{4},\frac{3}{4}]^2$ and $\Omega_{c_2}=[\frac{1}{4},\frac{3}{4}]\times[\frac{5}{4},\frac{7}{4}]$. We denote $\Omega_c := \Omega_{c_1} \cup \Omega_{c_2}$, which represents the poorly permeable part. Similarly, the diagonal entry of $\bfK$ is defined with its entry $10^{-2}$ in $\Omega_c$ and 1 elsewhere. The boundary conditions for the Darcy flow are imposed as follows:
    $$
    p = 1, \quad {\rm on}\ \{0\} \times (1,2),\qquad
    p = 0, \quad {\rm on}\ \{2\} \times (0,1),\qquad
    \bfu \cdot \bfn = 0, \quad {\rm elsewhere}.
    $$
    The pressure and the velocity are displayed in Figure \ref{ex4_fig1} with 384 elements. We can see from Figure \ref{ex4_fig2} that the residual (\ref{local_equal}) computed by the CG method is particularly larger than that computed by the EPG method at the vertices of $\Omega_{c_1}$ and $\Omega_{c_2}$. The number of elements on the mesh we test is 24576. Thus the performance of the local mass conservation property of the EPG method is robust. Furthermore, we test the transport equation numerically. The time step size is chosen to be 0.01. It can be demonstrated that the robustness of the EPG method in terms of the upper bound preserving property for the transport can also be guaranteed while the concentration for the CG method still overshoots, as shown in Figures \ref{ex4_fig3} - \ref{ex4_fig7}.

\begin{figure}[htbp]
    \includegraphics[width=3.5cm,height=3cm]{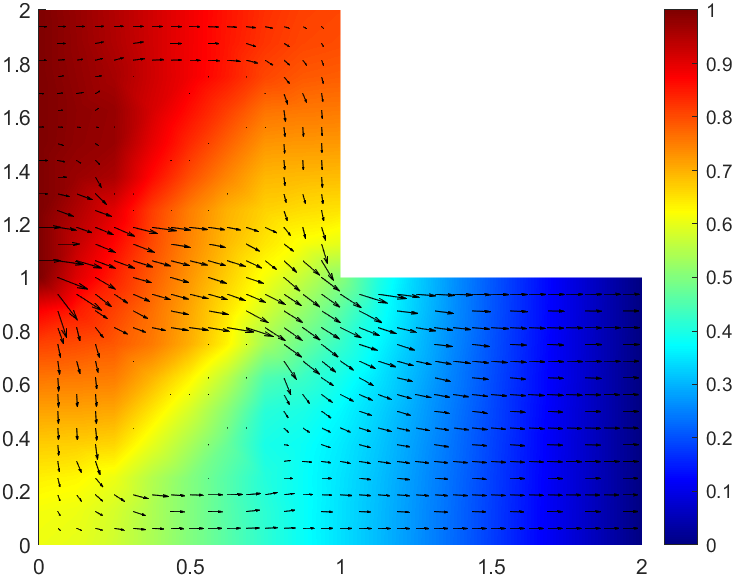}\qquad
    \includegraphics[width=3.5cm,height=3cm]{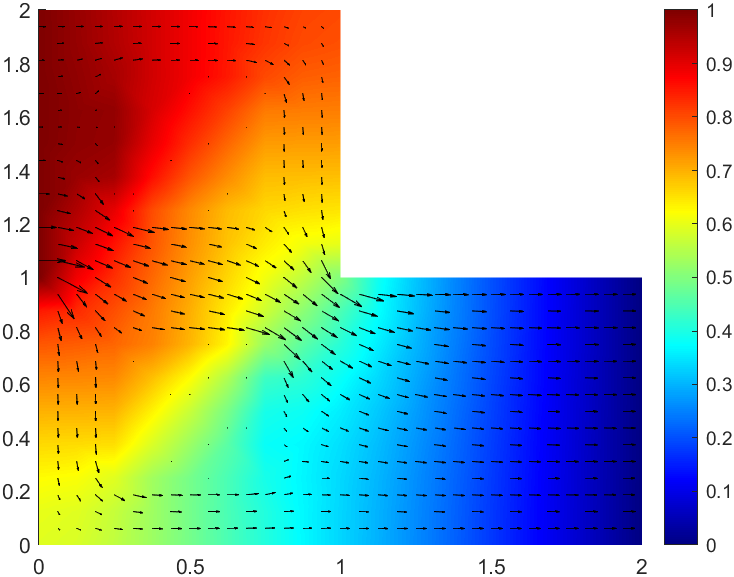}\qquad
    \includegraphics[width=3.5cm,height=3cm]{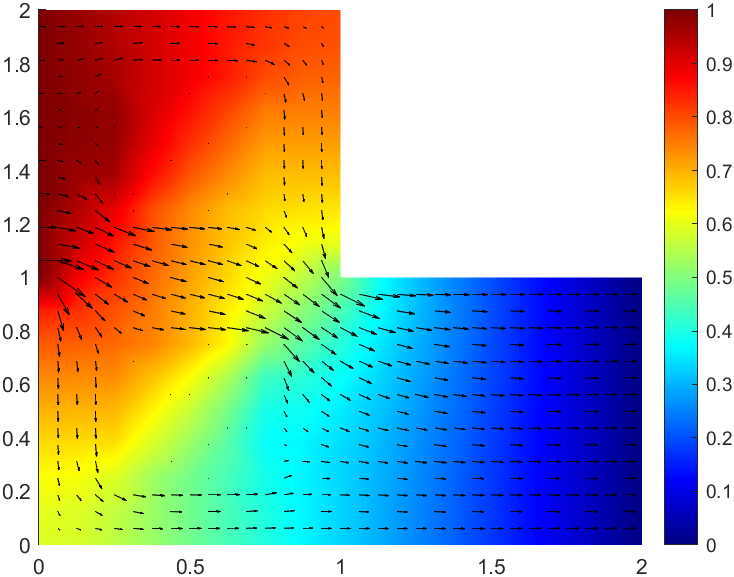}\\
    \includegraphics[width=3.5cm,height=3cm]{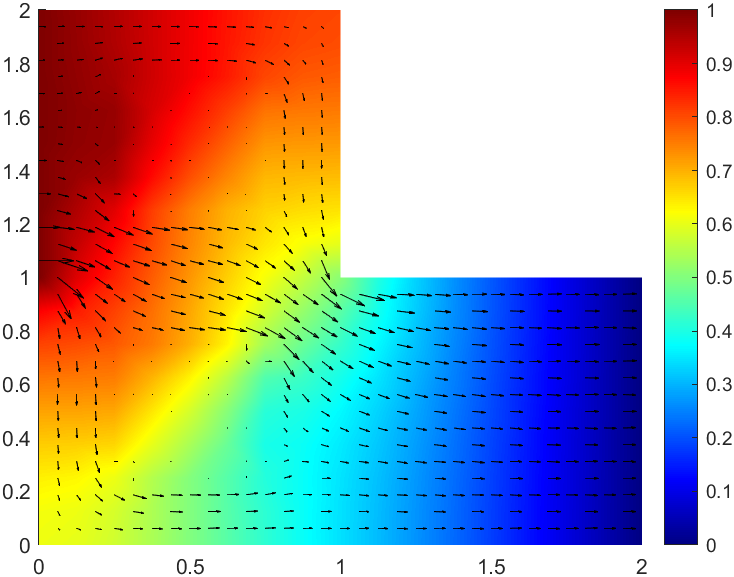}\qquad
    \includegraphics[width=3.5cm,height=3cm]{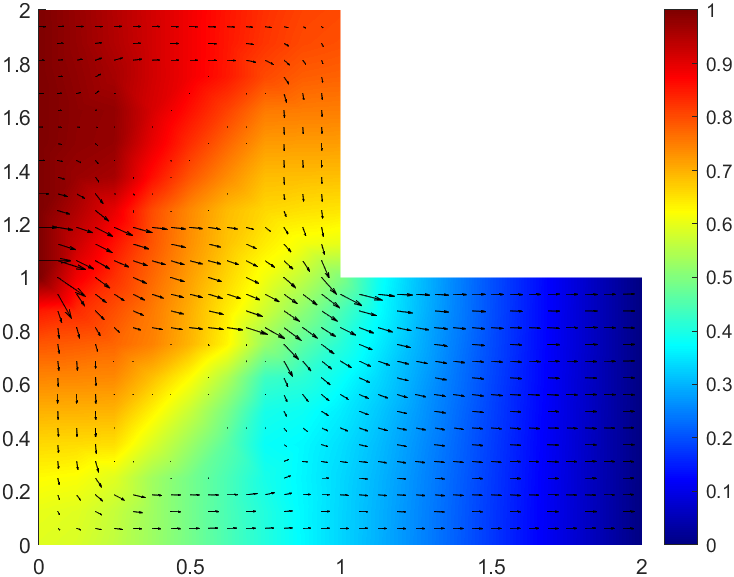}\qquad
    \includegraphics[width=3.5cm,height=3cm]{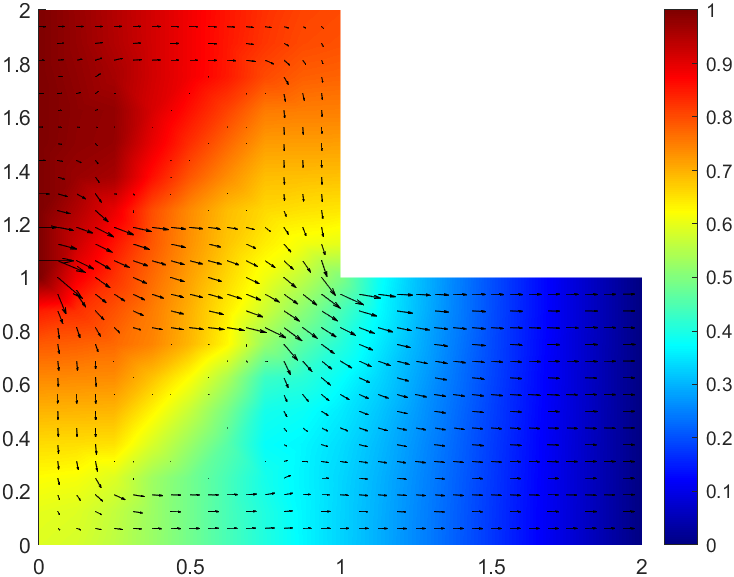}
    \caption{\footnotesize (Example \ref{example4})  The solutions of pressure and Darcy velocity. Top-left: P1-CG. Top-middle: P2-CG. Top-right: P3-CG. Bottom-left: P1-EPG. Bottom-middle: P2-EPG. Bottom-right: P3-EPG.}
    \label{ex4_fig1}
\end{figure}

\begin{figure}[htbp]
    \includegraphics[width=3.5cm,height=3cm]{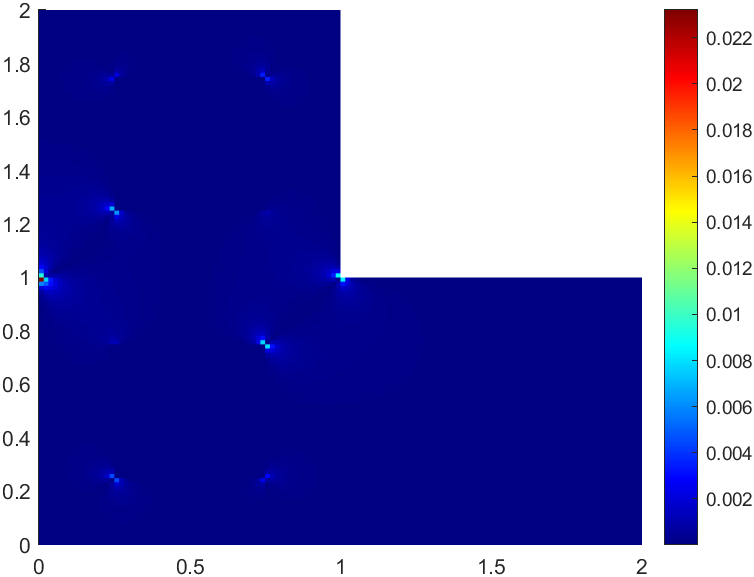}\qquad
    \includegraphics[width=3.5cm,height=3cm]{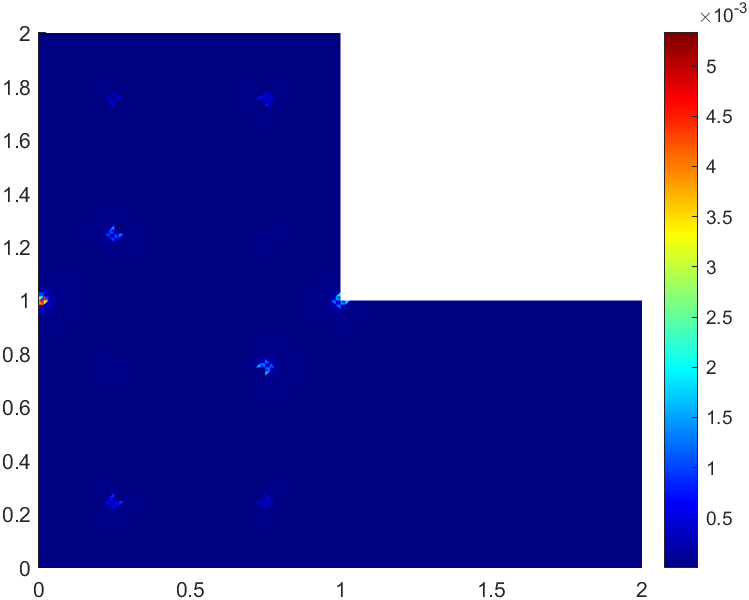}\qquad
    \includegraphics[width=3.5cm,height=3cm]{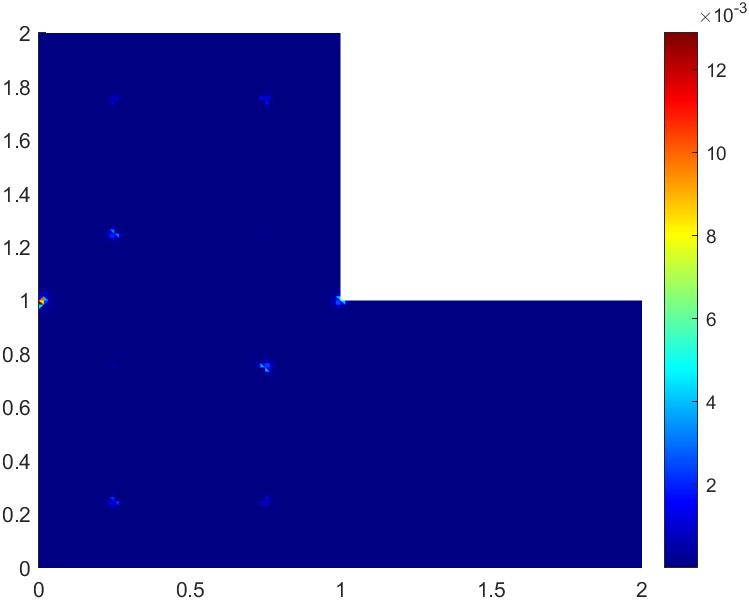}\\
    \includegraphics[width=3.5cm,height=3cm]{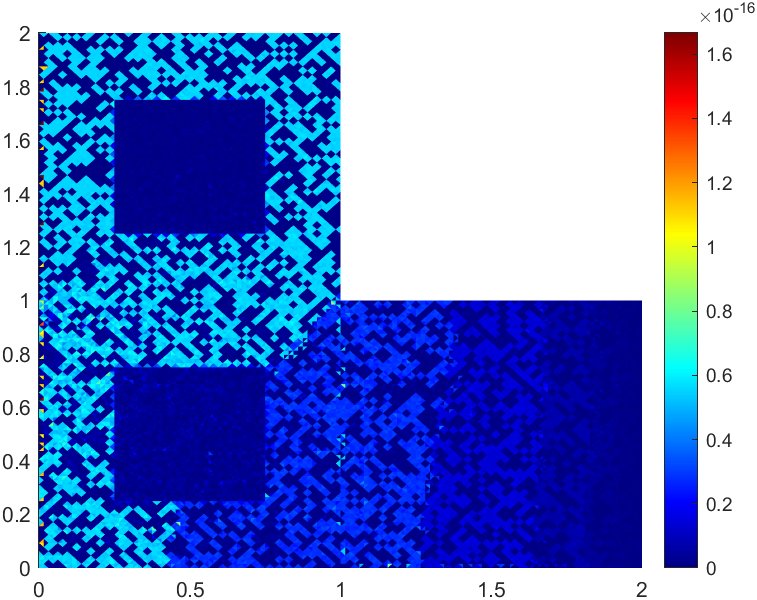}\qquad
    \includegraphics[width=3.5cm,height=3cm]{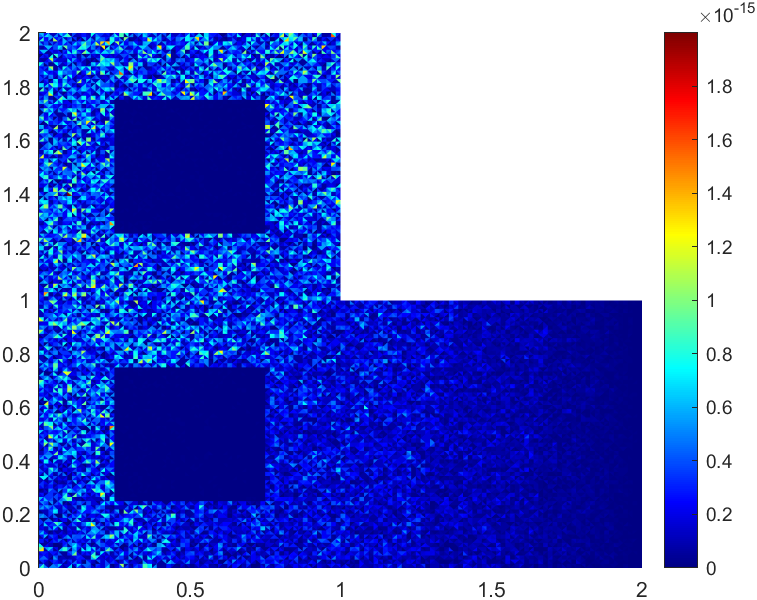}\qquad
    \includegraphics[width=3.5cm,height=3cm]{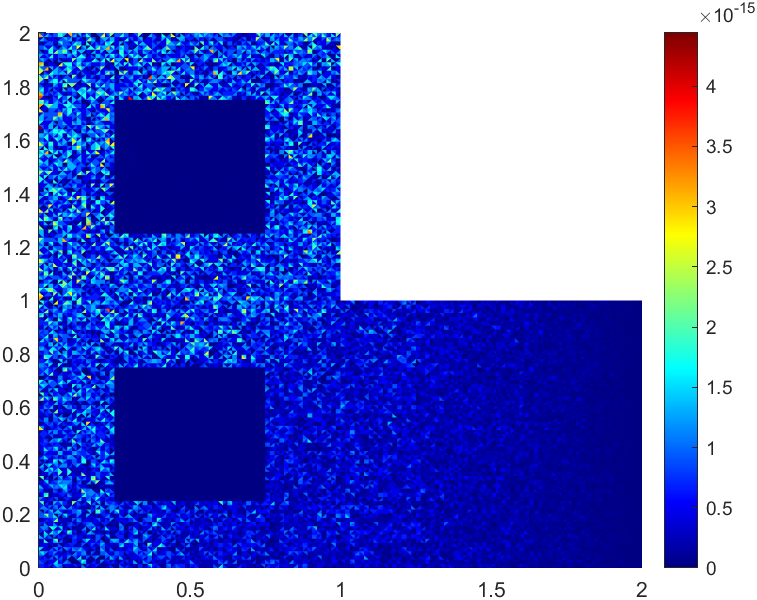}
    \caption{\footnotesize (Example \ref{example4})  Local mass conservation residual. Top-left: P1-CG. Top-middle: P2-CG. Top-right: P3-CG. Bottom-left: P1-EPG. Bottom-middle: P2-EPG. Bottom-right: P3-EPG.}
    \label{ex4_fig2}
\end{figure}

\begin{figure}[htbp]
    \includegraphics[width=3.5cm,height=3cm]{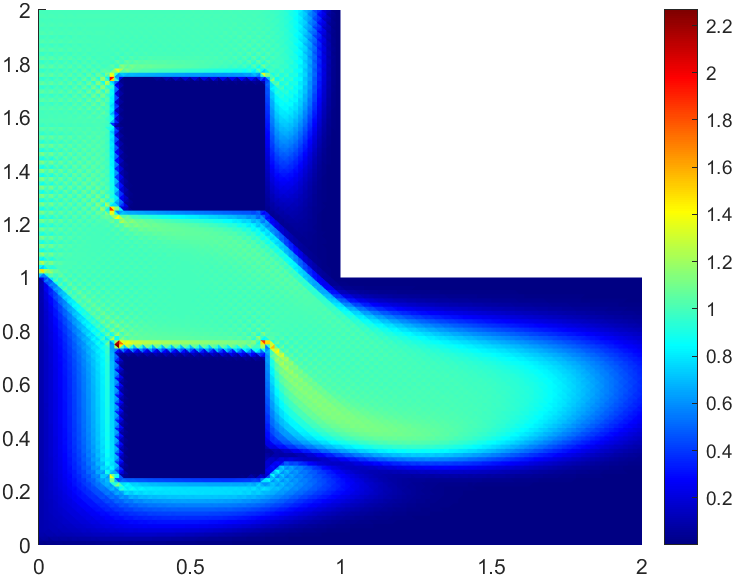}\qquad
    \includegraphics[width=3.5cm,height=3cm]{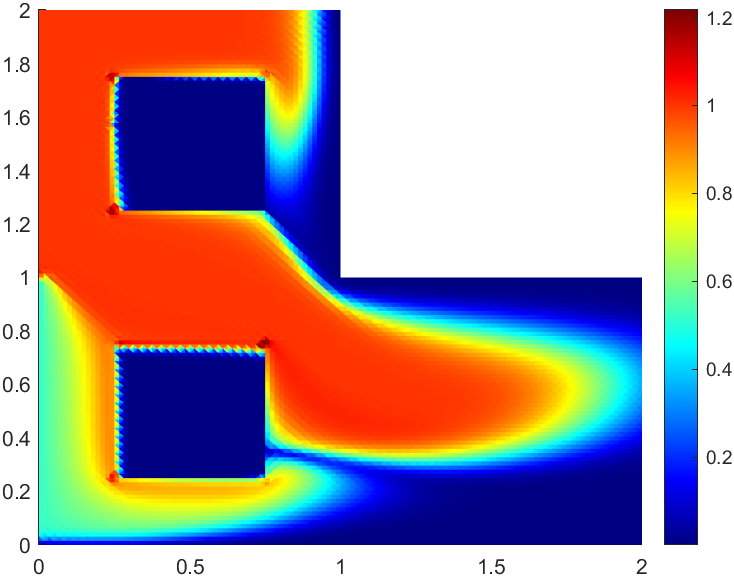}\qquad
    \includegraphics[width=3.5cm,height=3cm]{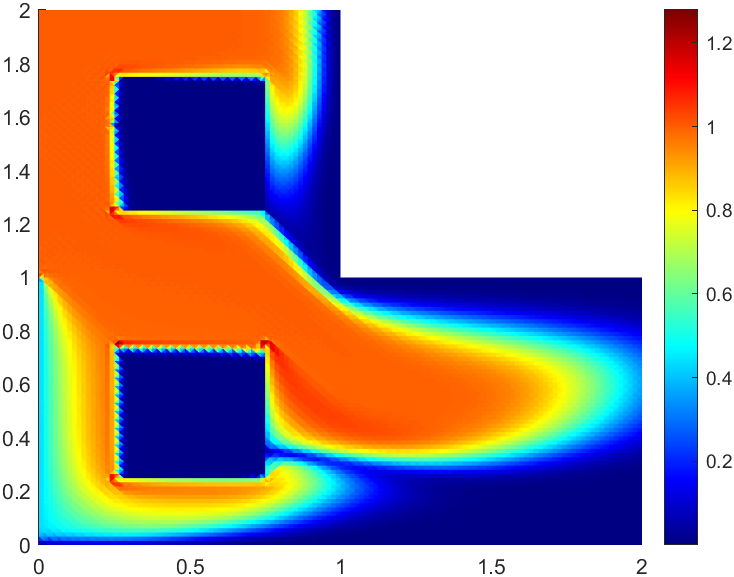}
    \caption{\footnotesize (Example \ref{example4})  Simulations of concentration at time 0.9 based on the velocity from the CG method. Left: P1-CG. Middle: P2-CG. Right: P3-CG.}
    \label{ex4_fig3}
\end{figure}

\begin{figure}[htbp]
    \includegraphics[width=3.5cm,height=3cm]{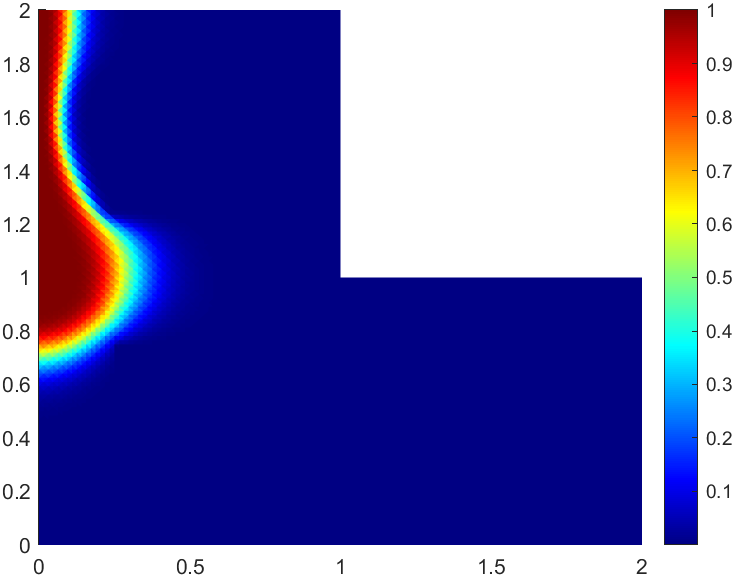}\qquad
    \includegraphics[width=3.5cm,height=3cm]{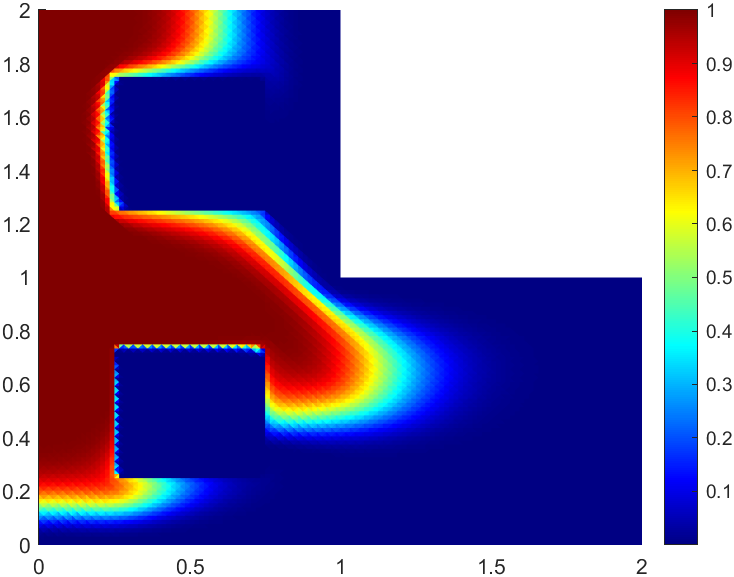}\qquad
    \includegraphics[width=3.5cm,height=3cm]{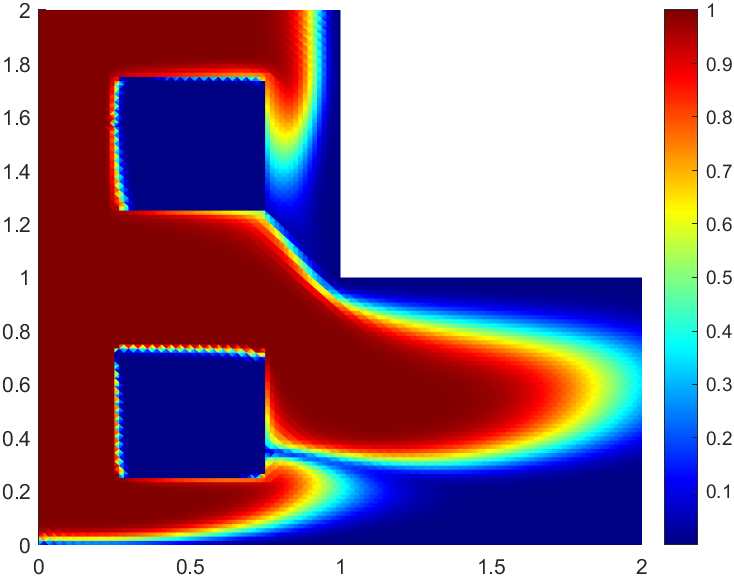}
    \caption{\footnotesize (Example \ref{example4})  Simulations of concentration based on the velocity from P3-EPG. Left-right: Simulations at time 0.1, 0.5 and 0.9.}
    \label{ex4_fig6}
\end{figure}

\begin{figure}[htbp]
    \includegraphics[width=3.5cm,height=3cm]{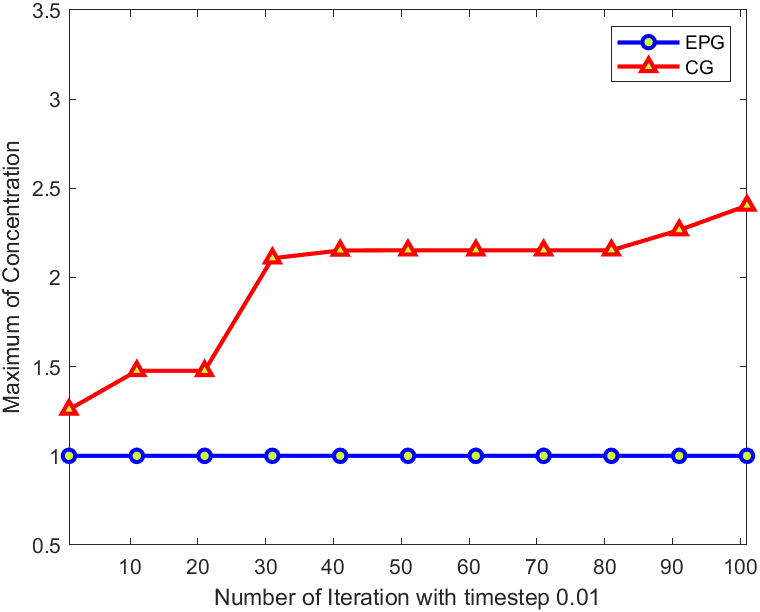}\qquad
    \includegraphics[width=3.5cm,height=3cm]{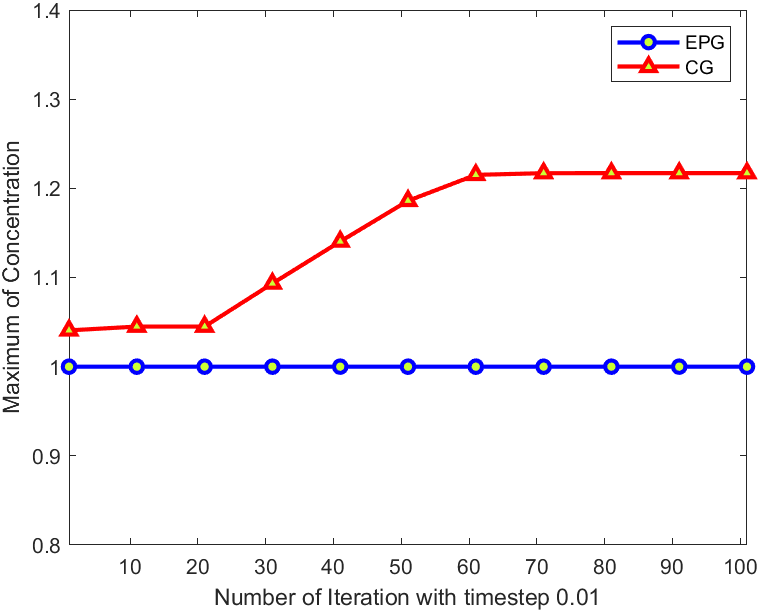}\qquad
    \includegraphics[width=3.5cm,height=3cm]{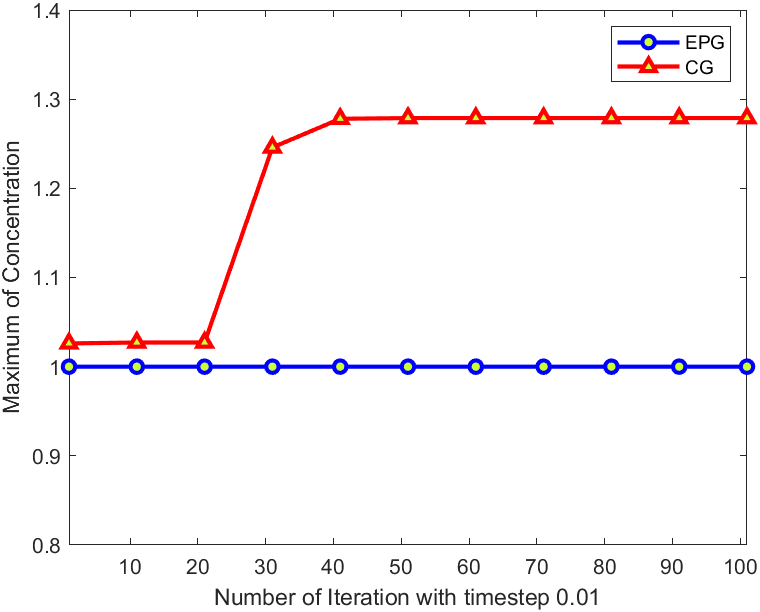}
    \caption{\footnotesize (Example \ref{example4})  Maximum concentrations based on the velocity from the CG and EPG methods. Left-right: P1, P2 and P3.}
    \label{ex4_fig7}
\end{figure}    
\end{example}

From the above three examples, we have the following observations: Firstly, for the Darcy flow with an exact solution such as Example \ref{example1}, we test and verify the optimal convergence in Figure \ref{ex1_fig2}. The results indicate that our proposed EPG method performs almost identical to the CG method in terms of the optimal convergence. Secondly, we test the local mass conservation of the CG and EPG methods according to (\ref{local_equal}). The computation is based on the Darcy flow and $\bfu_h$ here means the recovered velocity. From the Figures \ref{ex1_fig2}, \ref{ex3_fig2} and \ref{ex4_fig2}, we can see that the results obtained by calculating the maximum of locally mass-conservative residual for the CG method are much larger than those obtained for the EPG method. This confirms that the EPG method has the local mass conservation property that the CG method does not possess. Thirdly, we substitute the velocity computed from the Darcy flow into the species transport to numerically solve the concentration. We have already demonstrated the upper bound preserving property of the concentration in Theorem \ref{implicit thm}, which utilizes the local mass conservation of the velocity. To verify the upper bound preserving property, we observe the values of the maximum concentration at different iteration steps. The numerical results above show that the concentration computed based on the velocity from the CG method instead of from the EPG method typically overshoots.

%%%%%%%%%%%%%%%%%%%%%%%%%%%%%%%%%%%%%%%%%%%%%%%%%%%%%%%
%----------------section 6-----------------------------
\section{Conclusion}\label{section6}
%%%%%%%%%%%%%%%%%%%%%%%%%%%%%%%%%%%%%%%%%%%%%%%%%%%%%%%
In this paper we focus on the EPG method for the Darcy flow in porous media. The corresponding weak formulation we provide does not require any penalty term. Besides, the theoretical analysis and numerical experiments presented in this work have demonstrated various features of the EPG method, especially its efficiency and accuracy, as well as its preservation of the local mass conservation, as applied to the coupled flow and transport. We believe that our proposed method can also be extended to the flow problems in fractured porous media, which is our future investigation.

%%%%%%%%%%%%%%%%%%%%%%%%%%%%%%%%%%%%%%%%%%%%%%%%%%%%%%%
%------------------reference--------------------------
%%%%%%%%%%%%%%%%%%%%%%%%%%%%%%%%%%%%%%%%%%%%%%%%%%%%%%%
\providecommand{\href}[2]{#2}

\end{document}